\documentclass[onefignum,onetabnum]{siamart190516}

\usepackage[english]{babel}
\usepackage{amssymb,amsmath,mathtools}
\usepackage[latin1,utf8]{inputenc}
\usepackage{graphicx}
\usepackage{epstopdf}
\usepackage{subfigure}
\usepackage[section]{placeins}
\usepackage[]{algorithm}
\usepackage{algorithmic}

\usepackage{wrapfig}

\usepackage{tikz}
\usetikzlibrary{fit,positioning,arrows,automata}

\usepackage{cite}
\usepackage[utf8]{inputenc}
\usepackage[mathscr]{euscript}
\usepackage{hyperref}
\usepackage{color}
\usepackage{bm}

\usepackage{cleveref}
\usepackage{autonum}

\usepackage[normalem]{ulem}

\def\minwrt[#1]{\underset{#1}{\text{minimize }}}
\def\argminwrt[#1]{\underset{#1}{\text{arg min }}}
\def\maxwrt[#1]{\underset{#1}{\text{maximize }}}
\def\argmaxwrt[#1]{\underset{#1}{\text{arg max }}}
\def\maxemphwrt[#1]{\underset{#1}{\text{\emph{maximize} }}}
\newcommand{\ett}{{\bf 1}}

\newcommand{\parent}{p}

\newcommand{\mR}{{\mathbb R}}

\newcommand{\diag}{{\rm diag}}

\newtheorem{remark}{Remark}
\newtheorem{example}{Example}

\newcommand{\tr}{{\rm trace}}

\def\bC{{\bf C}}

\def\bK{{\bf K}}

\def\bM{{\bf M}}

\def\bU{{\bf U}}
\def\bV{{\bf V}}

\def\ccE{{\mathcal{E}}}

\def\ccG{{\mathcal{G}}}

\def\ccL{{\mathcal{L}}}
\def\ccM{{\mathcal{M}}}
\def\ccN{{\mathcal{N}}}

\def\ccP{{\mathcal{P}}}

\def\ccT{{\mathcal{T}}}
\def\ccV{{\mathcal{V}}}
\def\ccW{{\mathcal{W}}}

\def\RR{{\mathbb{R}}}

\begin{document}

\raggedbottom

\title{Multi-Marginal Optimal Transport\\ with a Tree-structured cost and\\ the Schr\"odinger Bridge Problem%
\thanks{This work was supported by the Swedish Research Council (VR), grant 2014-5870, KTH Digital Futures, SJTU-KTH cooperation grant, the NSF under grant 1901599 and 1942523, and Knut and Alice Wallenberg foundation under grant KAW 2018.0349.}}

\author{Isabel Haasler\thanks{Division of Optimization and Systems Theory, Department of Mathematics, KTH Royal Institute of Technology, Stockholm, Sweden. {\tt\small haasler@kth.se}, {\tt\small johan.karlsson@math.kth.se}}
\and Axel Ringh\thanks{Department of Electronic and Computer Engineering, The Hong Kong University of Science and Technology, Clear Water Bay, Kowloon, Hong Kong, China. {\tt\small eeringh@ust.hk}}
\and Yongxin Chen\thanks{School of Aerospace Engineering,
Georgia Institute of Technology, Atlanta, GA, USA. {\tt\small yongchen@gatech.edu}}
\and Johan Karlsson\footnotemark[2]} 
%

\maketitle

\begin{abstract}
The optimal transport problem has recently developed into a powerful framework for various applications in estimation and control.
Many of the recent advances in the theory and application of optimal transport are based on 
regularizing the problem with an entropy term, which connects it to 
the Schr\"odinger bridge problem and thus to stochastic optimal control.
Moreover, the entropy regularization makes 
the otherwise computationally demanding optimal transport problem feasible even for large scale settings. This has led to an accelerated development of optimal transport based methods in a broad range of fields. 
Many of these applications have an underlying graph structure, for instance information fusion and tracking problems can be described by trees.
In this work we consider multi-marginal optimal transport problems with a cost function that decouples according to a tree structure.
The entropy regularized multi-marginal optimal transport problem can be viewed as a generalization of the Schr\"odinger bridge problem with the same tree-structure, and by utilizing these connections we extend the computational methods for the classical optimal transport problem in order to solve structured multi-marginal optimal transport problems in an efficient manner. In particular, the algorithm requires only matrix-vector multiplications of relatively small dimensions. 
We show that the multi-marginal regularization 
introduces less diffusion, compared to the commonly used pairwise regularization, and is therefore more suitable for many applications.
Numerical examples illustrate this, and we finally apply the proposed framework for tracking of an ensemble of indistinguishable agents.

\end{abstract}


\section{Introduction}


An optimal transport problem is to find a transport plan that minimizes the cost of moving the mass of one distribution to another distribution\cite{villani2008optimal}. 
Historically this problem has been important in economics and operations research, but as a result of recent progress in the area 
it has become a popular tool in a wide range of fields such as control theory \cite{yang2017, chen2016optimalPartI, mikami2008optimal,Ghoussoub18, bayraktar2018martingale, acciaio2019extended}, signal processing \cite{elvander19multi,Kolouri17omt}, computer vision \cite{Dominitz10texture,solomon2015},  and machine learning \cite{Adler17inverse,ArjChiBot17,MonMulCut16}. 

An extension to the standard optimal transport framework is multi-marginal optimal transport \cite{pass2015multi}, which seeks a transport plan between not only two, but several distributions.
Early works on multi-marginal optimal transport include \cite{ruschendorf2002n, ruschendorf1995optimal, gangbo1998optimal}.
In this work we consider multi-marginal optimal transport problems with cost functions that decouple according to a tree structure. We refer to such a problem as a tree-structured multi-marginal optimal transport problem.
This should not be confused with optimal transport problems on graphs as in \cite{conforti2017reciprocal, chen2016robust, chow2012fokker}, where 
the distributions are defined over the nodes of the graphs.
Tree-structured cost functions generalize many structures that are commonly used in applications of optimal transport.
For example, 
a path tree naturally appears in tracking and interpolation applications \cite{chen2018state, solomon2015, Bonneel11}. Similarly, star trees are used in barycenter problems, which occur for instance in information fusion applications \cite{cuturi2014barycenter, elvander2018tracking}.

The optimal transport problem can be formulated as a linear program.
However, in many practical applications the problem is too large to be solved directly.
For the bi-marginal case,  these computational limitations have recently been alleviated by regularizing the problem with an entropy term \cite{cuturi2013sinkhorn}. The optimal solution to the regularized optimal transport problem can then be expressed in terms of dual variables, which can be efficiently found by an iterative scheme, called Sinkhorn iterations \cite{Sinkhorn67}.
For the multi-marginal case, although the Sinkhorn iterations can be generalized in a straight-forward fashion \cite{benamou2015bregman}, the complexity of the scheme increases dramatically with the number of marginals \cite{lin2019complexity}. Thus Sinkhorn iterations alone are not sufficient to address many multi-marginal problems. 
However, in specific settings, structures in the cost function can be exploited in order to 
derive computationally feasible methods, e.g., for Euler flows \cite{benamou2015bregman} and in tracking and information fusion applications \cite{elvander19multi}.

Unregularized multi-marginal optimal transport problems with transport cost that decouples according to a tree structure
can be equivalently formulated as a sum of coupled
 pairwise optimal transport problems, as for instance, tracking and interpolation problems \cite{elvander2018tracking, chen2018state, Bonneel11}, and barycenter problems \cite{agueh2011barycenters,cuturi2014barycenter,benamou2015bregman,solomon2015}. 
 Typically these problems are solved using pairwise regularization
(see, e.g., \cite[Sec.~3.2]{benamou2015bregman} and \cite{kroshnin2019complexity, lin2020revisiting, bonneel2016wasserstein}).
However, we have empirically observed in some applications that the multi-marginal formulation yields favourable solutions compared to a corresponding pairwise optimal transport estimate \cite{elvander19multi}.
One main contribution of
this work is to develop a framework for solving tree-structured optimal transport problems with a multi-marginal regularization.
For these problems, we show that the Sinkhorn algorithm can be performed in an efficient way, requiring only successive matrix-vector multiplications of relatively small size compared to that of the original multi-marginal problem. 
Thus we extend the computational results from \cite{elvander19multi} to general trees.

The entropy regularized formulation of optimal transport is connected to another classical topic, the Schr\"odinger bridge problem \cite{chen2016relation, chen2016hilbert, Leo12, leonard2013schrodinger, Mikami2004}. Schr\"odinger was interested in determining the most likely evolution of a particle cloud observed at two time instances, where the particle dynamics have deviated from the expected Brownian motion \cite{schrodinger1931}.
Schr\"odinger showed that this particle evolution can be characterized as the one out of all the theoretically possible ones, that minimizes the relative entropy to the Wiener measure. 
This optimal solution may be found by solving a so-called Schr\"odinger system, which turns out to be tightly connected to the Sinkhorn iterations for the entropy regularized optimal transport problem \cite{chen2016hilbert,leonard2013schrodinger}.
This framework has been used in robust and stochastic control problems \cite{Vladimirov15, chen2016optimalPartI}. 
A version of the Schr\"odinger bridge problem, that is discrete in both time and space, can be formulated by modeling the evolutions of a number of particles as a Markov chain \cite{pavon2010discrete, Georgiou2015discreteSB}. In the infinite particle limit, the maximum likelihood solution of the Markov process can then be approximated by solving a relative entropy problem \cite{chen2016robust}. 
An analogous approach has recently been used to develop a framework for modelling ensemble flows on Hidden Markov chains \cite{haasler19ensemble}. Other optimal transport based state estimation problems for a continuum of agents and in continuous time have been considered in \cite{chen2018state,chen2018measure,CheConGeoRip19}. For further work on ensemble controllability and observability see, e.g., \cite{zeng2019sample,chen2019structure}.

Based on \cite{haasler19ensemble}, we extend the  discrete Schr\"odinger bridge in \cite{pavon2010discrete} to trees. In particular, we derive a maximum likelihood estimate for Markov processes defined on the edges of a rooted directed tree. This leads to our second main result:
Interestingly, it turns out that the solution to the entropy regularized tree-structured multi-marginal optimal transport problem 
corresponds to the solution of the generalized Schr\"odinger bridge with the same tree-structure.
This generalizes the established equivalence of the classical bi-marginal entropy regularized optimal transport problem and the Schr\"odinger bridge problem \cite{leonard2013schrodinger}. Moreover, this also gives an additional motivation for using multi-marginal optimal transport instead of the often used pairwise optimal transport for tree structured problems. 

This paper is organized as follows:
Section~\ref{sec:background} is an introduction to the problems of optimal transport and Schr\"odinger bridges. 
The main results are presented in Section~\ref{sec:omt_tree} and \ref{sec:HMM_tree}.
In Section~\ref{sec:omt_tree} we define the tree-structured multi-marginal optimal transport problem, and provide an algorithm for efficiently solving its entropy regularized version.
In Section~\ref{sec:HMM_tree} we generalize the Schr\"odinger bridge problem to trees, and show that it is equivalent to an entropy regularized multi-marginal optimal transport problem on the same tree. 
Section~\ref{sec:pairwise} is on the discrepancy between the multi-marginal and pairwise optimal transport formulations with tree structure.
A larger numerical simulation is detailed in Section~\ref{sec:ensemble_flows}, where the proposed framework is used to estimate ensemble flows from aggregate and incomplete measurements.
Most of the proofs can be found in the Appendix, however due to space limitations two of the proofs are deferred to the supplementary material.

\section{Background} \label{sec:background}
In this section we summarize some background material on optimal transport and Schr\"odinger bridges. This is also used to set up the notation. To this end, first note that throughout we let $\exp( \cdot )$, $\log( \cdot )$, $\odot$, and $./$ denote the elementwise exponential, logarithm, multiplication, and division of vectors, matrices, and tensors, respectively. Moreover, $\otimes$ denotes the outer product. By $\ett$ we denote a column vector of ones, the size of which will be clear from the context.

\subsection{Optimal transport} \label{sec:omt}

In this work we consider the discrete optimal transport problem. For its continuous counterpart see, e.g., \cite{villani2008optimal}. 
Let the vectors $\mu_1, \mu_2 \in \RR_+^n$ describe two nonnegative distributions with equal mass.
The optimal transport problem is to find a mapping that moves the mass from $\mu_1$ to $\mu_2$, while minimizing the total transport cost.
Here the transport cost is defined in terms of a underlying cost matrix $C \in \RR_{+}^{n\times n}$, where $C_{i_1, i_2}$ denotes the cost of moving a unit mass from location $i_1$ to $i_2$.
Analogously, define a transport plan $M\in \RR^{n\times n}_+$, where $M_{i_1, i_2}$ describes the amount of mass that is moved from location $i_1$ to $i_2$.
An optimal transport plan from $\mu_1$ to $\mu_2$ is then a minimizing solution of
\begin{equation} \label{eq:omt_bi}
\begin{aligned}
T ( \mu_1, \mu_2 ) := \minwrt[M \in \RR^{n\times n}_+] & \tr(C^T M) \\
\text{ subject to } & M \ett = \mu_1,\quad M^T \ett = \mu_2 .
\end{aligned}
\end{equation}

Multi-marginal optimal transport extends the concept of the classical optimal transport problem \eqref{eq:omt_bi} to the setting with a set of marginals $\mu_1,\dots,\mu_J$, where $J\geq 2$ \cite{pass2015multi,benamou2015bregman,elvander19multi}. In this setting, the transport cost and transport plan in \eqref{eq:omt_bi} are described by J-mode tensors $\bC,\bM\in\RR^{n\times n \dots \times n}_+$. 
For a tuple $(i_1,\dots,i_J)$, the value $\bC_{i_{1},\dots,i_{J}}$ denotes the 
transport cost for a unit mass corresponding to the tuple, and similarly $\bM_{i_{1},\dots,i_{J}}$ represents the amount of transported mass associated with this tuple.
The multi-marginal optimal transport problem then reads
\begin{equation} \label{eq:omt_multi_discrete}
\begin{aligned}
\minwrt[ \bM \in \RR^{n\times \dots \times n}_+] & \langle \bC, \bM \rangle  \\
\text{ subject to } & P_j (\bM) = \mu_j,  \text { for } j \in \Gamma,
\end{aligned}
\end{equation}
where $\langle \bC, \bM \rangle = \sum_{i_1,\dots,i_J} \bC_{i_1,\dots,i_J} \bM_{i_1,\dots,i_J}$ is the standard inner product, the projection on the $j$-th marginal of $\bM$, denoted by $P_j(\bM)$, is defined  as
\begin{equation} \label{eq:proj_discrete}
P_j(\bM)_{i_j} = \sum_{i_1,\dots,i_{j-1},i_{j+1},\ldots, i_J} \bM_{i_1,\dots,i_{j-1},i_j,i_{j+1},\dots,i_J},
\end{equation}
and $\Gamma$ denotes an index set corresponding to the given marginals.
In the original multi-marginal optimal transport formulation, constraints are typically given on all marginals, i.e., for the index set $\Gamma = \{1,2,\dots,J\}$. However, in this work we consider the case, where constraints typically are only imposed on a subset of marginals, i.e., $\Gamma \subset \{1,2,\dots,J\}$.
Note that the standard bi-marginal optimal transport problem \eqref{eq:omt_bi} is a special case of the multi-marginal optimal transport problem \eqref{eq:omt_multi_discrete}, where $J=2$ and $\Gamma=\{1,2\}$.

\subsection{Entropy regularized optimal transport} \label{sec:entropy_reg}

Although linear, the number of variables in the multi-marginal optimal transport problem \eqref{eq:omt_multi_discrete} is often too large to be solved directly. 
A popular approach for the bi-marginal setting to bypass the size of the problem has been to add a regularizing entropy term to the objective. In theory the same approach can be used also for the multi-marginal case. 
\begin{definition}[\!\!{\cite[Ch.~4]{Gzyl95entropy}}] \label{def:kl}
	Let $p$ and $q$ be two nonnegative vectors, matrices or tensors of the same dimension. The normalized Kullback-Leibler (KL) divergence of $p$ from $q$ is defined as $H(p|q) :=	\sum_{i} \left(p_i \log\left( p_i /q_i \right) -p_i + q_i\right)$,
	where $0\log 0$ is defined to be $0$. Similarly, define $	H( p) := H(p| \ett) = \sum_{i} \left(p_i \log( p_i) -p_i +1\right)$, which is effectively the negative of the entropy of $p$.
\end{definition}
Note that $H(p|q)$ is jointly convex over $p, q$. For a detailed description of the KL divergence see, e.g.,\cite{cover2012elements,Cziszar91entropy}. 
The entropy regularized multi-marginal optimal transport problem is the convex problem
\begin{equation} \label{eq:omt_multi_regularized}
\begin{aligned}
\minwrt[ \bM \in \RR^{n\times \dots \times n}_+] & \langle \bC, \bM \rangle + \epsilon H(\bM) \\
\text{ subject to } & P_j (\bM) = \mu_j,  \text { for } j \in \Gamma,
\end{aligned}
\end{equation}
where $\epsilon>0$ is a regularization parameter.
For the bi-marginal case \eqref{eq:omt_bi}, where the cost and mass transport tensors are matrices, the entropy regularized problem reads
\begin{equation} \label{eq:omt_reg}
\begin{aligned}
T_\epsilon( \mu_1, \mu_2 ) := \minwrt[M \in \RR^{n\times n}_+] & \tr(C^T M) + \epsilon H(M) \\
\text{ subject to } & M \ett = \mu_1,\quad M^T \ett = \mu_2 .
\end{aligned}
\end{equation}
It is well established that the entropy regularized bi-marginal optimal transport problem is connected to the Schr\" odinger bridge problem \cite{chen2016relation, chen2016hilbert, leonard2013schrodinger, Mikami2004}, which is introduced in Section \ref{sec:schrodinger}.
More importantly from a computational perspective, the introduction of the entropy term in problem \eqref{eq:omt_reg} allows for expressing the optimal solution $M$ in terms of the Lagrange dual variables, which may be 
computed by Sinkhorn iterations\cite{cuturi2013sinkhorn}.
This procedure can be generalized to the setting of multi-marginal optimal transport \cite{benamou2015bregman}.

In particular, for the multi-marginal entropy regularized optimal transport problem \eqref{eq:omt_multi_regularized} it can be shown that the optimal solution is of the form\cite{elvander19multi}
\begin{equation} \label{eq:MKU}
\bM = \bK \odot \bU
\end{equation}
where $\bK = \exp(- \bC/\epsilon)$ and where $\bU$ can be decomposed as
\begin{equation}\label{eq:U}
	\bU= u_1 \otimes u_2 \otimes \dots \otimes u_J \quad \text{with } u_j = \begin{cases} \exp(  \lambda_j/\epsilon),& \text{ if } j \in \Gamma \\
	\ett,& \text{ else.} \end{cases}
	\end{equation}
Here $\lambda_j\in \mR^n$ for $j \in \Gamma$ are optimal dual variables  in the dual problem
of \eqref{eq:omt_multi_regularized}: 
\begin{equation} \label{eq:multi_omt_dual}
\maxwrt[\lambda_j\in \mR^n,\, j\in \Gamma] - \epsilon \langle\bK, \bU \rangle + \sum_{j \in \Gamma} \lambda_j^T \mu_j,
\end{equation}
where $\bU$ depends on the variables $\lambda_j$ as specified by \eqref{eq:U}. 
Note that the dual variable $\lambda_j$
corresponds to the constraint on the $j$-th marginal.
For details the reader is referred to, e.g., \cite{elvander19multi, benamou2015bregman}.

The Sinkhorn scheme for finding $\bU$ in \eqref{eq:U},  
is to iteratively update $u_j$ as 
\begin{equation} \label{eq:sinkhorn_multi}
u_j \leftarrow u_j \odot \mu_j ./ P_j(\bK \odot \bU),
\end{equation}
for all $j\in\Gamma$. This scheme may for instance be derived as Bregman projections \cite{benamou2015bregman} or a block coordinate ascend in the dual \eqref{eq:multi_omt_dual}, \cite{ringh2017sinkhorn,elvander19multi,tseng1990dual}. 
As a result, global convergence of the Sinkhorn scheme \eqref{eq:sinkhorn_multi} is guaranteed \cite{BauLew00}.
For the sake of completeness, we also provide a result on the linear convergence rate in our presentation (see Theorem~\ref{thm:convergence}).
We also note that \eqref{eq:sinkhorn_multi} reduces to standard Sinkhorn iterations, 
\begin{equation} \label{eq:sinkhorn_bimarginal}
u_1 \leftarrow \mu_1./(K u_2),\quad u_2 \leftarrow \mu_1./ (K^T u_1),
\end{equation}
for the two-marginal case \eqref{eq:omt_reg}. 
The iterations \eqref{eq:sinkhorn_bimarginal} converge linearly to an optimal solution $u_1,u_2$, which is unique up to multiplication/division with a constant \cite{chen2016hilbert, franklin1989}. 

The computational bottleneck of the Sinkhorn iterations \eqref{eq:sinkhorn_multi} is computing the projections $P_j(\bM)$, for $j \in \Gamma$, which in general scales exponentially in $J$. 
In fact, even storing the tensor $\bM$ is a challenge as it  consists of $n^J$ elements.
However, in many cases of interest, structures in the cost tensors can be exploited to make the computation of the projections feasible. In \cite{elvander19multi} this is shown on the example of cost functions that decouple sequentially, centrally, or a combination of both. Computing the projections requires then only repeated matrix vector multiplications. In this work, we show that 
these efficient methods can be generalized to the setting where
the cost tensor decouples according to a tree structure, that is, when the marginals of the optimal transport problem are associated with the nodes of a tree, and cost matrices are defined on its edges (see Section \ref{sec:omt_tree}).

\subsection{Schr\"odinger bridge problem} \label{sec:schrodinger}
The Schr\"odinger bridge problem is to determine the most likely evolution of a particle cloud observed at two time instances \cite{schrodinger1931}. In the case that the second observation cannot be explained as a Brownian motion of the initially observed particle cloud, Schr\"odinger aimed to find the most likely particle evolution connecting, hence bridging, the two distributions.

A useful mathematical framework to solve this problem, however not yet developed in Schr\"odingers time, is the theory of large deviations, which studies so called rare events, meaning deviations from the law of large numbers \cite{dawson1990schrodinger,dembo2009large}. The probability of these rare events approaches zero, as the number of trials goes to infinity, and large deviation theory analyzes the rate of this decay, which can often be expressed as the exponential of a so called rate function \cite{ellis2006book, dembo2009large}.

For a large deviation interpretation of the Schr\"odinger bridge \cite[Sec.~II.1.3]{Follmer88}, the particle evolutions are
modelled as independent identically distributed random variables on path space, and the Schr\"odinger bridge is the probability measure $\ccP$ on path space that is most likely to describe the rare event of observing the two particle distributions. Let $\ccW$ be the Wiener measure, which corresponds to the probability law of the Brownian motion. Then the Schr\"odinger bridge $\ccP$ is found by minimizing the corresponding rate function $\int \log(d\ccP/d\ccW) d\ccP$ over all probability measures that are absolutely continuous with respect to $\ccW$ and have the given particle cloud distributions as marginals.
A space and time discrete Schr\"odinger bridge problem for Markov chains is treated in \cite{pavon2010discrete, Georgiou2015discreteSB}.

We follow the exposition in \cite{haasler19ensemble}, where the space and time discrete Schr\"odinger bridge has been derived as a maximum likelihood approximation for Markov chains.
Consider a cloud of $N$ particles and assume that each particle evolves according to a Markov chain. 
Denote the states of the Markov chain by $X={\left\{X_1,X_2,\dots,X_n\right\}}$ and let the transition probability matrix be $A^j \in \RR^{n\times n}_+$, with elements $A^j_{k\ell} = P(q_{j+1}=X_\ell | q_{j}=X_k)$, where $q_j$ denotes the state at time $j$. 
Let the vector $\mu_j$ describe the particle distribution over the discrete state space $X$ at time $j$, for $j=1,\dots,J$.
As in the optimal transport framework we define the mass transport matrix $M^{j}$ where element $M^{j}_{k\ell}$ describes the number of particles transitioning from state $k$ at time $j$ to state $\ell$ at time $j+1$.

Given the distribution $\mu_j$, and transition probability matrix $A^j$, a large deviation argument similar to the continuous setting can be used to approximate the most likely transfer matrix $M^j$ by the minimizer of the rate function $H(\, \cdot \, | \, \diag(\mu_j) A^j)$.
Assuming that the initial and final marginal, $\mu_1$ and $\mu_J$, are given, these large deviation theoretic considerations motivate the formulation of an optimization problem to find the most likely mass transfer matrices $M^j$, for $j=1,\dots,J-1$, and intermediate distributions $\mu_j$, for $j=2,\dots,J-1$, as
\begin{equation} \label{eq:opt_markov_chain}
\begin{aligned}
\minwrt[M_{[1:J-1]}, \mu_{[2:J-1]}] \ & \sum_{j=1}^{J-1} H( M^j\,|\,\diag(\mu_j)A^j) \\
\text{subject to } \  &  M^j \mathbf{1} = \mu_j  , \; \;  (M^j)^T \mathbf{1} = \mu_{j+1}  , \; \text{ for } \  j=1,\dots,J-1.
\end{aligned}
\end{equation}
Indeed, this problem is equivalent to the discrete time and space Schr\"odinger bridge problem in \cite{pavon2010discrete} (for details see \cite{haasler19ensemble} or Section~ \ref{sec:ex_bridge}).

An extension to a Hidden Markov Model formulation with aggregate indirect observations of the distributions was described in \cite{haasler19ensemble}. In this article the framework is extended to general tree structures, where each vertex is associated with a distribution, and each edge with a Markov process.

\section{ Tree-structured multi-marginal optimal transport} \label{sec:omt_tree}

In this section we introduce tree-structured multi-marginal optimal transport problems. For the entropy regularized version of these problems, the projections \eqref{eq:proj_discrete} can be computed by only matrix-vector multiplications. This yields an efficient Sinkhorn algorithm, which we present in this section. Moreover, some properties of tree-structured multi-marginal optimal transport problems are discussed.
\begin{definition}
	A graph $\ccT=(\ccV,\ccE)$, with vertices $\ccV$ and edges $\ccE$, is a tree if it is acyclic and connected \cite{diestel10graph}.
	The vertices with degree 1 are called leaves, and we denote the set of leaves by $\ccL$.
	For a vertex $j\in\ccV$, the set of neighbours $\ccN_j$ is defined as the set of vertices, which have a common edge with $j$.
	The path between two vertices $j_1$ and $j_L$ is a sequence of edges that connect $j_1$ and $j_L$, such that all edges are distinct. We also denote a path by the set of intermediate vertices.
	\end{definition}

Let $\ccT=(\ccV, \ccE)$ be a tree. In this section we consider a multi-marginal optimal transport problem where the marginals correspond to the nodes of the tree, i.e., $\ccV = \{1,2,\dots,J\}$. Assume that the cost tensor decouples as
\begin{equation} \label{eq:cost_tensor_tree}
\bC_{i_{1},\dots,i_{J}} = \sum_{(j_1,j_2)\in \ccE} C^{(j_1,j_2)}_{i_{j_1},i_{j_2}},
\end{equation}
where $C^{(j_1,j_2)}\in \mR^{n\times n}_+$ is a cost matrix characterizing the cost of transportation between marginal $j_1$ and $j_2$, for all $(j_1,j_2)\in\ccE$.
Since we consider an undirected tree, the expression \eqref{eq:cost_tensor_tree} should not depend on the ordering of the indices $(j_1,j_2)$ in the edges. This is achieved by simply letting $C^{(j_2,j_1)}= \left(C^{(j_1,j_2)}\right)^T$ for $(j_1,j_2) \in \ccE$.
 We refer to problem \eqref{eq:omt_multi_discrete} with a cost of the form \eqref{eq:cost_tensor_tree} as a multi-marginal optimal transport problem on the tree $\ccT$.

In the following, we consider discretized and entropy regularized multi-marginal optimal transport with tree structure. The transport plan can then be expressed as $\bM = \bK \odot \bU$, where $\bK=\exp(-\bC/\epsilon)$ (cf. \eqref{eq:MKU}). Due to the structured cost \eqref{eq:cost_tensor_tree}, this tensor decouples as
\begin{equation}\label{eq:Kmultimarginal}
\bK_{i_{1},\dots,i_{J}} = \prod_{(j_1,j_2)\in \ccE} K^{(j_1,j_2)}_{i_{j_1},i_{j_2}},
\end{equation}
with the matrices defined as $K^{(j_1,j_2)}_{i_{j_1},i_{j_2}}= \exp( -C^{(j_1,j_2)}_{i_{j_1},i_{j_2}}/ \epsilon)$, for $(j_1,j_2)\in\ccE$.
Thus, it holds $K^{(j_2,j_1)}= \left(K^{(j_1,j_2)}\right)^T$ for $(j_1,j_2)\in \ccE$.
We show that for these problems the projections \eqref{eq:proj_discrete} can therefore be computed by successive matrix vector multiplications. As the computation of the projections account for the bottleneck of the Sinkhorn iterations \eqref{eq:sinkhorn_multi}, this yields an efficient algorithm for solving entropy regularized multi-marginal optimal transport problems with tree structure.

We first illustrate the computation of the projections on a small example.
%
%
	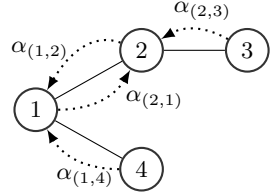
\begin{wrapfigure}{r}{0.28\textwidth}
		\small 
		\tikzstyle{level 1}=[level distance=40pt, sibling distance=45pt]
		\tikzstyle{level 2}=[level distance=40pt, sibling distance=45pt]
		\centering
		\begin{tikzpicture}[grow=right]
		\tikzstyle{main}=[circle, minimum size = 5mm, thick, draw =black!80]	
		\node[main] (1) {$1$}
		child{ node[main] (4) {$4$}}
		child{ node[main] (2) {$2$}
			child{ node[main] (3) {$3$}  }  };
		\draw [->, -latex, dotted, thick] (2) to [out=160, in = 60] node[ left]{$ \alpha_{(1,2)}$} (1);
		\draw [->, -latex, dotted, thick ] (1) to [out=0, in = 240] node[ right]{$\ \ \alpha_{(2,1)}$} (2);
		\draw [->, -latex, dotted, thick ] (3) to [out=150, in = 30] node[ above]{$\ \ \alpha_{(2,3)}$} (2);
		\draw [->, -latex, dotted, thick ] (4) to [out=180, in = 300] node[ below]{$\ \ \alpha_{(1,4)}$} (1);
		\end{tikzpicture}
		\vspace{-5pt}
		\caption{Illustration of the tree in Example \ref{ex:tree_multi_omt}.}  \vspace{-10pt}
		\label{fig:small_tree}
	\end{wrapfigure}
%
\begin{example} \label{ex:tree_multi_omt}
	Consider the tree $\ccT=(\ccV,\ccE)$ with vertices $\ccV=\{1,2,3,4\}$ and edges $\ccE=\{(1,2),(2,3),(1,4)\}$ as depicted in Figure~\ref{fig:small_tree}.
Assume that the cost in problem \eqref{eq:omt_multi_regularized} decouples as
$\bC_{i_1,i_2,i_3,i_4} = C^{(1,2)}_{i_1,i_2} + C^{(2,3)}_{i_2,i_3} + C^{(1,4)}_{i_1,i_4}$, where $C^{(j_1,j_2)}$, for $(j_1,j_2)\in\ccE$, are cost matrices defined on the respective edges. The transport tensor is then of the form $ \bM = \bK \odot \bU$, where $\bK=\exp(-\bC/\epsilon)$ and $\bU=u_1\otimes u_2 \otimes u_3 \otimes u_4$. Thus, denoting $K^{(j_1,j_2)}= \exp(-C^{(j_1,j_2)}/\epsilon)$ it holds that $\bK_{i_1,i_2,i_3,i_4} =  K^{(1,2)}_{i_1,i_2} K^{(2,3)}_{i_2,i_3} K^{(1,4)}_{i_1,i_4}$. Hence, the projection on the first marginal of $\bK \odot \bU$ can be computed as
 \begin{equation}
P_1(\bK \odot \bU)_{i_1} \!\! = \!\!\!\!  \sum_{i_2,i_3,i_4} \!\! \!  \bK_{i_1,i_2,i_3,i_4} \bU_{i_1,i_2,i_3,i_4} \!\! = \! (u_1)_{i_1} \!\!\!\! \sum_{i_2,i_3,i_4} \!\! \!\! K^{(1,2)}_{i_1,i_2} K^{(2,3)}_{i_2,i_3} K^{(1,4)}_{i_1,i_4}  (u_2)_{i_2} (u_3)_{i_3} (u_4)_{i_4}.
\end{equation}
The sum can be computed as successive matrix-vector multiplications, starting from the leaves of the tree. 
In particular, denote the products starting in the two leaves
\begin{equation}
	\sum_{i_3}\! K^{(2,3)}_{i_2,i_3} \! (u_3)_{i_3} \!=\! \left( K^{(2,3)} u_3 \right)_{i_2} \!\!\!\!=: \!(\alpha_{(2,3)})_{i_2}\!,\ \ \ \ \sum_{i_4}\! K^{(1,4)}_{i_1,i_4} \! (u_4)_{i_4} \!=\! \left( K^{(1,4)} u_4 \right)_{i_1} \!\!\!\!=: \!\left(\alpha_{(1,4)}\right)_{i_1}\! \!.
\end{equation}
On the lower branch of the tree, we have thereby brought the expression to a vector indexed by $i_1$. On the upper branch, another multiplication, corresponding to the edge $(1,2)$, is required. This leads to defining
\begin{equation}
  \sum_{i_2}  K^{(1,2)}_{i_1,i_2} \left( u_2 \odot \alpha_{(2,3)}\right)_{i_2}
  =  \left( K^{(1,2)} ( u_2 \odot \alpha_{(2,3)} ) \right)_{i_1}
  =:\left(\alpha_{(1,2)}\right)_{i_1}.
\end{equation}
The full expression for the projection thus reads
\begin{equation}
P_1(\bK \odot \bU) = u_1 \odot K^{(1,2)}  ( u_2 \odot K^{(2,3)} u_3 ) \odot K^{(1,4)} u_4 
= u_1 \odot \alpha_{(1,2)} \odot \alpha_{(1,4)}.
\end{equation}
This thus allows for decomposing the computation of $P_1(\bK \odot \bU)$
into the three parts $u_1,  \alpha_{(1,2)}, \alpha_{(1,4)}$,
which represents the contributions from node 1; from nodes 2 and 3; and from node 4, respectively. 
Analogously, the projection on the second marginal is computed as 
$P_2(\bK \odot \bU) = u_2 \odot \alpha_{(2,3)} \odot \alpha_{(2,1)}$, where $ \alpha_{(2,1)} := K^{(2,1)} \left(u_1 \odot \alpha_{(1,4)} \right)$.
In particular, note here that 
the tuple $(j,k)$ in $\alpha_{(j,k)}$ is ordered. For instance, it holds $ \alpha_{(1,2)} \neq \alpha_{(2,1)}$.
The other marginals can be computed similarly by performing the corresponding matrix vector products, starting from the leaves of the tree.
\end{example}
Illustratively, the procedure in Example \ref{ex:tree_multi_omt} for projecting the information of the full tensor down to one marginal can be understood as passing down the information from all branches of the tree to the desired vertex, where we interpret the vectors $u_j$ as local information in each node. Starting from the leaves of the tree, this is done according to the following rules:
\begin{enumerate}
	\item Information is passed down the edge $(j_1,j_2)\in\ccE$ by multiplication with the matrix $K^{(j_1,j_2)}$.
	\item Information is collected in the node $j\in\ccV$ by elementwise multiplication of the vector $u_j$, which contains local information, with information passed down from the connected branches.
\end{enumerate}
In this interpretation the vector $\alpha_{(j,k)}$ can be understood as the collected information that is sent from node $k$ to node $j$.
The following theorem shows that these rules
hold for any tensor of the form $\bK\odot\bU$, with $\bU=u_1\otimes \dots \otimes u_j$ and $\bK$ decoupling according to an arbitrary tree.
\begin{theorem} \label{thm:multi_omt_tree}
	Let  $\bK=\exp(-\bC /\epsilon)$ with $\bC$ as in  \eqref{eq:cost_tensor_tree} for the tree $\ccT=(\ccV,\ccE)$, and let $\bU= u_1 \otimes u_2 \otimes \dots \otimes u_J$. Define $K^{(j_1,j_2)}= \exp(-C^{(j_1,j_2)}/\epsilon)$, for $(j_1,j_2)\in\ccE$. Then the projection on the $j$-th marginal of $\bK \odot \bU$ is of the form
	\begin{equation}
	P_j(\bK \odot \bU) =  u_j \odot \bigodot_{k\in \ccN_j} \alpha_{(j,k)}.
	\end{equation}
	The vectors $\alpha_{(j,k)}$, for all ordered tuples $(j,k)\in \ccE$, can be computed recursively starting in the leaves of the tree according to
\begin{equation}	\label{eq:alpha}
	\begin{aligned} 
	\alpha_{(j,k)} &= K^{(j,k)} u_k, &\mbox{ for } k\in \ccL,\\
	\alpha_{(j,k)} &= K^{(j,k)} \Big( u_k \odot \bigodot_{ \ell \in \ccN_k \setminus \{j\} } \alpha_{(k,\ell)} \Big), &\mbox{ for } k\notin \ccL.
	\end{aligned}
\end{equation}
\end{theorem}

\begin{proof}
See the Appendix.
\end{proof}

Analogously to the marginal projections \eqref{eq:proj_discrete}, let $P_{j_1,j_2}(\bM)$ denote the bi-marginal projections on marginals $j_1$ and $j_2$ as $ P_{j_1,j_2}(\bM)_{i_{j_1}, i_{j_2}} = \sum_{ i_1,\dots,i_J \setminus \{i_{j_1}, i_{j_2}\} } \bM_{i_1,\dots,i_J}$.
These pairwise projections can be expressed similarly to the marginal projections in Theorem~\ref{thm:multi_omt_tree}.
\begin{proposition} \label{prp:multi_omt_tree_pairwise}
	Let the assumptions in Theorem~\ref{thm:multi_omt_tree} hold, let $j_1,j_L \in\ccV$, and let $j_1,j_2,\dots,j_L$ denote the path between $j_1$ and $j_L$.
		Then the pairwise projection of $\bK \odot \bU$ on the marginals $j_1$ and $j_L$, denoted by $P_{j_1,j_L}(\bK\odot \bU) $, is
	\begin{equation} \label{eq:proj_j}
 \Bigg( \prod_{\ell=1}^{L-1}  \diag\bigg(u_{j_\ell} \odot \!\!\!  \!\!\! \bigodot_{ \substack{k\in \ccN_{j_\ell} \\ k \notin \{j_1,\dots,j_L\}} } \!\!\! \!\!\! \alpha_{(j_\ell,k)}  \bigg)  K^{(j_{\ell},j_\ell+1)} \Bigg)  \diag\bigg( u_{j_{L}} \odot \bigodot_{k \in \ccN_{j_L} \setminus \{j_{L-1}\}} \alpha_{(j_L,k)} \bigg),
	\end{equation}
	where $\alpha_{(j_1,j_2)}$, for $(j_1,j_2)\in\ccE$, are defined as in Theorem~\ref{thm:multi_omt_tree}
\end{proposition}

\begin{proof}
See the supplementary material.
\end{proof}

Multi-marginal optimal transport problems on some specific trees have previously been introduced in \cite{elvander19multi}. It is worth noting that the expressions for the projections of the mass tensor $\bM$ for the examples in \cite{elvander19multi} are special cases of the results in Theorem~\ref{thm:multi_omt_tree} and Proposition~\ref{prp:multi_omt_tree_pairwise}.
For instance, in tracking problems each marginal in the optimal transport problem is associated with a time instance, and the corresponding graph is a path graph, as in \cite[Section 3.1 and 5.1]{elvander19multi}. Sensor fusion applications can be cast as barycenter problems, which are described by star-shaped graphs, as in \cite[Section 3.2 and 5.2]{elvander19multi}. Moreover, a combination of these two applications is the tracking of barycenters over time, which is described by a graph as in Figure~\ref{fig:HMM_tree_model}, see \cite[Section 3.3 and 5.3]{elvander19multi}.

The expressions for the projections in Theorem~\ref{thm:multi_omt_tree} can be used to solve a multi-marginal optimal transport problem, with cost structure according to the tree $\ccT=(\ccV,\ccE)$, by a Sinkhorn method as in \eqref{eq:sinkhorn_multi}.
To this end, without loss of generality, we consider tree-structured multi-marginal optimal transport problems \eqref{eq:omt_multi_discrete}, where the constraints are given on the set of leaves, i.e., $\Gamma=\ccL$. The following proposition shows that if a marginal is given for a non-leaf node, then the problem can be decomposed into smaller problems on subtrees, where the marginals are known exactly on the set of leaves of the subtrees.
\begin{proposition} \label{prp:GammaL}
	Consider problem \eqref{eq:omt_multi_regularized} with constraint set $\Gamma$ and cost structured according to the tree $\ccT$ with leaves $\ccL$. The following holds:
	\begin{enumerate}
		\item If there is a marginal $k \in \Gamma$, but $k \notin \ccL$, then the solution to \eqref{eq:omt_multi_regularized} can be found by solving problems of form \eqref{eq:omt_multi_regularized} on the subtrees of $\ccT$, which are obtained by cutting $\ccT$ in node $k$.
		\item If there is a marginal $\ell \in \ccL$, but $\ell \notin \Gamma$, then the solution to \eqref{eq:omt_multi_regularized} can be found by solving the problem on the subtree of $\ccT$ that is obtained by removing the node $\ell$ and its adjoining edge from $\ccT$.
	\end{enumerate}
\end{proposition}

\begin{proof}
See the Appendix.
\end{proof}

Note that since $\Gamma=\ccL$, the factors $u_j$, for $j\in\ccV \setminus \ccL$, can be neglected in Theorem~\ref{thm:multi_omt_tree} and Proposition~\ref{prp:multi_omt_tree_pairwise}.
Moreover, when applying the iterative scheme \eqref{eq:sinkhorn_multi}, in each iteration step some of the factors $\alpha_{(j_1,j_2)}$ do not change. Specifically, between the update of $u_{j_1}$ and $u_{j_2}$ only the factors on all edges that lie on the path between nodes $j_1$ and $j_2$ need to be updated.
The full method for solving tree-structured multi-marginal optimal transport problems is summarized in Algorithm~\ref{alg:sinkhorn}. 
Moreover, the algorithm has in fact linear convergence speed.
\begin{algorithm}[tb]
\begin{algorithmic}
\STATE Given: Tree $\ccT=(\ccV,\ccE)$ with leaves $\ccL=\{1,2,\dots,|\ccL|\}$;\\
Initial guess $u_j$, for $j\in \ccL$ 
\FOR{ all ordered tuples $(j_1,j_2)\in\ccE$}
\STATE Initialize $\alpha_{(j_1,j_2)}$ according to \eqref{eq:alpha}
\ENDFOR
\STATE Initialize $j\in\ccL$
\WHILE{Sinkhorn not converged}
\STATE $u_{j} \leftarrow  \mu_j ./ \bigodot_{k\in \ccN_j} \alpha_{(j,k)}$
\FOR{ $(j_1,j_2)\in\ccE$ on the path from $j$ to $(j+1 \mod |\ccL|)$}
\STATE  Update $\alpha_{(j_1,j_2)}$ according to \eqref{eq:alpha}
\ENDFOR
\STATE $j \leftarrow j+1 \mod |\ccL|$
\ENDWHILE
\RETURN $u_j$ for $j\in\ccL$
\end{algorithmic}
\caption{Sinkhorn method for the tree-structured multi-marginal optimal transport problem.}\label{alg:sinkhorn} 
\end{algorithm} 
\begin{theorem} \label{thm:convergence}
	Let $\{ u^k_j \}_{j \in \ccL}$ be the set of vectors after the $k$th iteration in the while-loop in Algorithm~\ref{alg:sinkhorn}. Then the sequence $\{ u^k_j \}_{j \in \ccL}$ converges at least linearly to an optimal solution of \eqref{eq:multi_omt_dual}, as $k \to \infty$.
\end{theorem}
\begin{proof}
Algorithm~\ref{alg:sinkhorn} implements the Sinkhorn iterations \eqref{eq:sinkhorn_multi}, which are a block coordinate ascent in the dual multi-marginal optimal transport problem \eqref{eq:multi_omt_dual} \cite{elvander19multi}. Thus, the result follows from Application 5.3 and Theorem 2.1 in \cite{luo1992convergence}.
\end{proof}

It should be noted that computing the marginals $P_j(\bK \odot \bU)$ by summing over the indices as defined in \eqref{eq:proj_discrete} scales exponentially in the number of marginals $J$ of the tensor $\bK \odot \bU $. Thus, the complexity of one Sinkhorn update in \eqref{eq:sinkhorn_multi} is in general of the order $\mathcal{O}(n^J)$. Algorithm~\ref{alg:sinkhorn} utilizes the result in Theorem~\ref{thm:multi_omt_tree} in order to exploit the graph-structure in the tensor $\bK \odot \bU$ to compute the projections. This decreases the computational complexity of the Sinkhorn iterations substantially.
In particular, if the direct path between node $j-1$ and node $j$ is of length $p$, then the update of the vector $u_j$ requires $p$ matrix-vector multiplications. The complexity of one update is thus of the order $\mathcal{O}(pn^2)$.
As a rule of thumb the leaves in the underlying tree $\ccT$ should thus be labeled such that the paths between any two nodes that are updated successively are short.
In case there are additional structures in the cost matrices $C^{(j_1,j_2)}$, for $(j_1,j_2)\in\ccE$, e.g., they describe the squared Euclidean distance, the cost of the matrix-vector multiplications can be decreased further, see, e.g., \cite[Rem.~4]{elvander19multi}, \cite[Rem.~3.11]{ringh2017sinkhorn}.

Note that Proposition~\ref{prp:GammaL}.1 implies that if the tensor $\bK$ is normalized to be a discrete probability distribution, it defines a Markov random field (see also \cite{haasler2020pgm}).
Every Markov process is also a reciprocal process \cite{jamison1974reciprocal}, which were first introduced in the context of studying Schr\"odinger bridges \cite{jamison75}. 
Recall that the maximum likelihood estimation problem for a Markov chain in Section~\ref{sec:schrodinger} is equivalent to a discrete Schr\"odinger bridge.
This suggests a connection between the multi-marginal optimal transport problem and the maximum likelihood estimation problem in Section~\ref{sec:schrodinger}.
In fact, under certain conditions a generalization of the Markov chain problem \eqref{eq:opt_markov_chain} to trees, which is introduced in Section~\ref{subsec:HMM_tree}, yields the same solution as the regularized multi-marginal optimal transport problem \eqref{eq:omt_multi_regularized} on the same tree, as we will see in Section~\ref{subsec:equivalence}.

\section{Multi-marginal optimal transport vs. Schr\"odinger bridge with tree-structure}
 \label{sec:HMM_tree}

In this section, we extend the Schr\"odinger bridge problem which is defined on a path tree to a formulation valid on an arbitrary tree graph.  
We then  show that under certain conditions this problem yields a solution which is equivalent to the solution of the entropy regularized multi-marginal optimal transport problem \eqref{eq:omt_multi_regularized} on the same tree, thus extending the well-known equivalence from the optimal transport problem and the Schr\"odinger bridge problem on a path tree \cite{chen2016relation, chen2016hilbert, leonard2013schrodinger, Mikami2004}.

\subsection{Tree-structured Schr\"odinger bridges}

\label{subsec:HMM_tree}

The notion of time of the Markov chain introduces a direction to the Schr\"odinger bridge problem, which can be seen directly in the optimization problem \eqref{eq:opt_markov_chain}. This directionality needs to be taken into account when extending the problem to a general tree. 
To this end, we introduce the following notation: Consider a tree $\ccT_r = (\ccV,\ccE_r)$ that is rooted in a vertex $r\in\ccL$, i.e., one of the leaves is defined to be 
the root of the tree. This defines a partial ordering on the tree, and we write $k<j$ 
if node $k$ lies on the path between $r$ and node $j$. 
To formulate the following results, we assume that all edges are directed according to this partial ordering, i.e., 
$j_1<j_2$ for all $(j_1,j_2)\in \ccE_r$.
For a vertex $j\in \ccV\setminus r$, its parent is then defined as the (unique) vertex $\parent(j)$ such that $(\parent(j),j)\in \ccE_r$. In the following, without loss of generality we denote the root vertex by $r=1$.

Let $A^{(j_1,j_2)}$ be the probability transition matrix on edge $(j_1,j_2)\in \ccE_r$. Then, the Schr\"odinger bridge problem in \eqref{eq:opt_markov_chain} may be naturally extended to the tree structure as
\begin{equation} \label{eq:HMM_tree}
\begin{aligned}
\minwrt[ \substack{M^{(j_1,j_2)}, \, (j_1,j_2) \in \ccE_r,\\ \mu_{j}, \,j \in \ccV \setminus \Gamma}] \ & \! \sum_{(j_1,j_2)\in \ccE_r} H \left( M^{(j_1,j_2)}\,|\,\diag(\mu_{j_1})A^{(j_1,j_2)} \right) \\
\text{subject to } \  &  M^{(j_1,j_2)} \mathbf{1} = \mu_{j_1}  , \ \  (M^{(j_1,j_2)})^T \mathbf{1} = \mu_{j_2}  ,\ \ \text{for } \  (j_1,j_2) \in \ccE_r.
\end{aligned}
\end{equation}
Moreover,
we can assume without loss of generality that $\Gamma=\ccL$, by similar reasoning as in Section~\ref{sec:omt_tree}.  
The extension of \eqref{eq:opt_markov_chain} to \eqref{eq:HMM_tree} builds on the large deviation principle in \cite[Prop. 1]{haasler19ensemble}, which requires that an initial distribution is known.
This is why we restrict our analysis to the case when the directed tree $\ccT_r(\ccV,\ccE_r)$ is rooted in a leaf.

Similarly to the multi-marginal optimal transport problem, the optimal solution to \eqref{eq:HMM_tree} can be expressed in terms of its dual variables.
In fact this is a natural generalization of the forward and backward propagation factors in the Schr\"odinger system  \cite{pavon2010discrete, Georgiou2015discreteSB} to tree graphs.

\begin{theorem} \label{thm:HMM_tree_sol}

Let $A^{(j_1,j_2)}$ be probability transition matrices, for $(j_1,j_2)\in \ccE_r$, with strictly positive elements and assume that $\ett^T \mu_{j_1}=\ett^T \mu_{j_2}$ for all leaves $j_1, j_2\in \ccL$.  Then an optimal solution to \eqref{eq:HMM_tree} can be written as
\begin{equation}
	\begin{aligned}
\mu_{j} =& \ \varphi_j \odot \hat\varphi_j,\\
M^{(j_1,j_2)} =& \  \diag\left( \hat \varphi_{j_1}  \odot \varphi_{j_1 \setminus j_2}  \right) A^{(j_1,j_2)} \diag \left(  \varphi_{j_2} \right),
\end{aligned}
\end{equation}
for a set of vectors $\varphi_j$, $\hat\varphi_j$, for $j \in \ccV$, and  $\varphi_{j_1 \setminus j_2}$ for $(j_1, j_2) \in \ccE_r$.
   
	Moreover, there exists
	a set of vectors $v_j$, for $j\in \Gamma$, such that the vectors $\varphi_j$, $\hat \varphi_j$ and $\varphi_{j_1 \setminus j_2}$ can be written as
	\begin{equation} \label{eq:phi_bw}
	\varphi_j = \begin{dcases}  v_j, & \text{ if } j\in\ccL\setminus\{1\} \\
	\bigodot_{k:(j,k) \in \ccE_r}  A^{(j,k)} \varphi_k ,& \text{ else,}
	\end{dcases}
	\end{equation}
	\begin{equation} \label{eq:phi_fw}
	\hat \varphi_j = \begin{dcases} \ett ./ v_{1},& \text{ if } j=1 \\
	(A^{(\parent(j),j)}) ^T ( \hat \varphi_{\parent(j)} \odot \varphi_{\parent(j) \setminus j}  ), & \text{ else,}  \end{dcases}
	\end{equation}
	and $\varphi_{j_1 \setminus j_2} = \varphi_{j_1} ./ \left( A^{(j_1,j_2)} \varphi_{j_2} \right)$.
\end{theorem}

\begin{proof}
See the Appendix.
\end{proof}

Note that the vectors $v_j$ for $j\in\ccL$ in Theorem~\ref{thm:HMM_tree_sol} satisfy the nonlinear equations
$ \varphi_1./v_1=\mu_1$, and $v_j \odot \hat \varphi_j=\mu_j$, for $j\in\ccL\setminus\{1\}$.
When considering the Schr\"odinger bridge, which is the special case with two leaves, these equations correspond to the boundary conditions in the Schr\"odinger system \cite[Thm.~4.1]{pavon2010discrete}.
The factors $\varphi_j$ and $\hat \varphi_j$ can be seen as propagating information from the leaves to the other nodes in a similar manner as $\alpha_{(j,k)}$ in Theorem~\ref{thm:multi_omt_tree}. 

Next, we show that the tree-structured Schr\"odinger bridge problem \eqref{eq:HMM_tree} is independent of the choice of root $r\in \ccL$. This follows as a corollary to the next proposition.

\begin{proposition} \label{prp:HMM_tree_anyroot}
	Let $\ccT_r=(\ccV,\ccE_r)$ be a rooted directed tree with root $r \in \ccL$.
	Then problem \eqref{eq:HMM_tree} is equivalent to the problem
	\begin{equation} \label{eq:HMM_tree_anyroot}
	\begin{aligned}
	\minwrt[ \substack{M^{(j_1,j_2)}, (j_1,j_2) \in \ccE_r,\\ \mu_{j}, j \in \ccV \setminus \Gamma}] \ & \sum_{(j_1,j_2)\in \ccE_r} H \left( M^{(j_1,j_2)}\,|\,A^{(j_1,j_2)} \right) - \sum_{j \in \ccV \setminus \ccL} (\deg(j)-1) H(\mu_j) \\
	\text{subject to } \  &  M^{(j_1,j_2)} \mathbf{1} = \mu_{j_1}  , \ \  (M^{(j_1,j_2)})^T \mathbf{1} = \mu_{j_2}  ,\ \ \text{for } \  (j_1,j_2) \in \ccE_r.
	\end{aligned}
	\end{equation}
\end{proposition}

\begin{proof}
See the Appendix.
\end{proof}
	
	\begin{corollary} \label{cor:root_independent}
	Let $\ccT_r=(\ccV,\ccE_r)$ be a directed tree with root $r\in \ccL$, and $ \ccT_{\hat r} =(\ccV,  \ccE_{\hat r})$ be a directed tree with root $ \hat r \in \ccL$ with the same strucure as $\ccT_r$ and the edges on the path between $r$ and $\hat r$ reversed.
	Then, the solution to \eqref{eq:HMM_tree} on $\ccT_r$ and on $ \ccT_{\hat r}$ are equivalent in the sense that the marginal distributions $\mu_j$, for $j\in\ccV$, and transport plans $M^{(j_1,j_2)}$, for $ (j_1,j_2) \in \ccE_r \cap \ccE_{ \hat r}$, are the same, and on the reversed edges it holds $M^{(j_2,j_1)}= (M^{(j_1,j_2)})^T$.
\end{corollary}

\begin{proof}
See the supplementary material.
\end{proof}

\subsection{Equivalence of problems (\ref{eq:omt_multi_regularized}) and (\ref{eq:HMM_tree}) for tree-structures} \label{subsec:equivalence}

We will now verify that the generalized Schr\"odinger bridge \eqref{eq:HMM_tree} on a rooted tree $\ccT_r$ is equivalent to an entropy regularized multi-marginal optimal transport problem \eqref{eq:omt_multi_regularized} on an undirected tree $\ccT$ with the same structure as $\ccT_r$.
In particular, given positive transition probability matrices $A^{(j_1,j_2)}$ for $(j_1,j_2)\in \ccE_r$ and a regularization parameter $\epsilon$, there is a natural choice of cost matrices $C^{(j_1,j_2)}$ for $(j_1,j_2)\in \ccE$, so that the minimizers to both problems represent the same solution.
For the original bi-marginal optimal transport problem \eqref{eq:omt_bi} the equivalence between entropy regularized
optimal transport
and the Schr\"odinger bridge problem has been studied extensively \cite{chen2016relation, chen2016hilbert, leonard2013schrodinger, Mikami2004}.
\begin{remark}[cf.{ \cite[Rem.~1]{haasler19ensemble}}] \label{rem:equivalence_bimarginal}
In the discrete setting, it is easy to see that with a cost matrix defined as $C= - \epsilon\log(A)$, the entropy regularized optimal transport problem \eqref{eq:omt_reg} is equivalent to the bi-marginal Schr\"odinger bridge \eqref{eq:opt_markov_chain}, by noting that the objective can then be written as
\begin{equation}
\tr(C^TM) + \epsilon H(M) = \sum_{i,j=1}^n  \epsilon  \left( M_{ij}  \log\left( \frac{M_{ij} }{ A_{ij} } \right) -M_{ij} +1 \right)  =  \epsilon H\left(M \, |\, A \right).
\end{equation}
Both bi-marginal problems thus yield the same optimal solution $M$.
\end{remark}
We will show that this equivalence holds true even for the multi-marginal case with tree-structured cost.
\begin{theorem} \label{thm:equivalence}
	Let $\ccT_r=(\ccV,\ccE_r)$ be a rooted directed tree with root $r=1 \in\ccL$, let $\ccT=(\ccV,\ccE)$ be its undirected counterpart, and let $\mu_j$ for $j \in\ccL$ be given marginals. 
Let $\epsilon>0$ and $C^{(j_1,j_2)}$ be such that 
	\begin{equation}
	K^{(j_1,j_2)}= A^{(j_1,j_2)}, \quad \text{ for } (j_1,j_2)\in \ccE_r.
	\end{equation}
Let $\bM = \bK\odot \bU$ be the solution to the entropy regularized multi-marginal optimal transport problem \eqref{eq:omt_multi_regularized} on $\ccT$, where $\bK$ and $\bU$ are defined as in Theorem~\ref{thm:multi_omt_tree}. Moreover, let $\mu_j$ for $j\in \ccV$, and $M^{(j_1,j_2)}$ for $(j_1,j_2)\in \ccE_r$, be the solution to the generalized Schr\"odinger bridge \eqref{eq:HMM_tree} on $\ccT_r$. 
Then,
\begin{equation}
	\begin{aligned}
		P_j(\bM) &= \mu_j, \qquad \qquad\;\, \text{ for } j\in \ccV,\\
	P_{(j_1,j_2)}(\bM) &= M^{(j_1,j_2)}, \qquad \text{ for } (j_1,j_2)\in \ccE_r.
	\end{aligned}
\end{equation}
Furthermore, if $v_j$ for $j\in \ccL$ are the corresponding vectors in Theorem~\ref{thm:HMM_tree_sol}, 
then it holds that
\begin{equation}
u_j = \begin{cases} \ett ./ v_1, & \text{if } j=1 \\ v_j, & \text{otherwise.} \end{cases}
\end{equation}
\end{theorem}

\begin{proof}
See the Appendix.
\end{proof}

In particular, given a maximum likelihood problem \eqref{eq:HMM_tree} structured according to the tree $\ccT_r=(\ccV,\ccE_r)$, a corresponding multi-marginal optimal transport problem can be formulated by definining the cost matrices as
$C^{(j_1,j_2)} = - \epsilon  \log(A^{(j_1,j_2)})$, for all $(j_1,j_2)\in\ccE_r$.
Vice versa, given an entropy regularized multi-marginal optimal transport problem \eqref{eq:omt_multi_regularized} with cost structured according to the tree $\ccT=(\ccV,\ccE)$, a corresponding maximum likelihood problem can be formulated by 
transforming the undirected tree into a rooted tree $\ccT_r=(\ccV,\ccE_r)$, where $r\in\ccL$,
and defining the matrices 
$A^{(j_1,j_2)} =  \exp \left( - C^{(j_1,j_2)}/ \epsilon \right)$,
for all $(j_1,j_2)\in\ccE_r$.
However, if these matrices are not row-stochastic, they cannot be interpreted as transition probability matrices of a Markov process.
\begin{remark}
For a nonnegative matrix $A$ with at least one positive element in each row, define the vector $b = \ett ./ A\ett$. Then the matrix $\hat A = \diag( b) A$ is row stochastic. Note that we can write $H \left(M \,|\, \diag(\mu) A \right) = H \left(M \,|\, \hat A \right) + H( \mu \, | \, b )$.
For a given rooted tree $\ccT_r=(\ccV,\ccE_r)$ assume that there is a set of matrices $A^{(j_1,j_2)}$, for $(j_1,j_2)\in\ccE_r$, which are not necessarily row-stochastic, but are such that for every $j_1 \in \ccE_r$ there is a vector $b_{j_1}$ such that it holds that $b_{j_1} = \ett ./ A^{(j_1,j_2)} \ett$ for all $j_2$ such that $(j_1,j_2)\in\ccE_r$. Then, according to Proposition~\ref{prp:HMM_tree_anyroot}, the objective function of the maximum likelihood estimation problem of a Markov process on $\ccT_r$ in \eqref{eq:HMM_tree} may be written as
\begin{equation} \label{eq:A_bethe}
\sum_{(j_1,j_2)\in \ccE_r} H \left( M^{(j_1,j_2)}\,|\,\hat A^{(j_1,j_2)} \right)  - \sum_{j \in \ccV \setminus \ccL} (\deg(j)-1) H(\mu_j \,|\, b_j).
\end{equation}
\end{remark}

The Sinkhorn scheme~\eqref{eq:sinkhorn_multi} can equivalently be formulated in terms of the notation from the maximum likelihood estimation problem in Theorem~\ref{thm:HMM_tree_sol}. Given an initial set of positive vectors $v_j$ for $j\in\ccL$, the updates are then expressed as
\begin{equation} \label{eq:sinkhorn_hmm}
\begin{aligned}
v_{1} =& \varphi_1 ./ \mu_1 \\
v_j =& \mu_j ./ \hat \varphi_j , \text{ for } j\in \ccL\setminus 1,
\end{aligned}
\end{equation}
where $\varphi_1$ is computed as in \eqref{eq:phi_bw} before each update of $v_1$, and $\hat \varphi_j$ is computed as in \eqref{eq:phi_fw} before updating $v_j$ for each $j\in \ccL\setminus 1$. This can be done efficiently, and by storing and reusing intermediate results appropriately, this algorithm is equivalent to Algorithm~\ref{alg:sinkhorn}.
It was shown that the classical Sinkhorn iterations are a block coordinate ascent in a Lagrange dual problem \cite{ringh2017sinkhorn,tseng1990dual}. In this sense, the scheme \eqref{eq:sinkhorn_hmm} can be understood as Sinkhorn iterations even for the problem \eqref{eq:HMM_tree}.

\begin{proposition} \label{prp:HMM_sinkhorn}
	Let $v_j$, for $j\in\ccL$, be an initial set of positive vectors. Iteratively performing the updates \eqref{eq:sinkhorn_hmm} corresponds to a block coordinate ascent in a Lagrange dual problem to \eqref{eq:HMM_tree}.
\end{proposition}

 In the light of Proposition~\ref{prp:HMM_sinkhorn} it is not surprising, that the algorithm presented in \cite{haasler19ensemble}, for the special case where the tree represents a hidden Markov chain, is of the form \eqref{eq:sinkhorn_hmm}, where the vectors $\varphi_j$ and $\hat\varphi_j$ are constructed according to the special tree structure.

\begin{remark} \label{rem:cycle}
The tree-structured optimal transport problem, i.e., problem \eqref{eq:omt_multi_discrete}
where the cost is of the form \eqref{eq:cost_tensor_tree} for a tree-graph $\ccG=(\ccV,\ccE)$, can be formulated as a sum of pairwise optimal transport costs 
	\begin{equation} \label{eq:omt_pairwise_unreg}
	\minwrt[  \mu_{j}, j \in \ccV \setminus \Gamma] \  \sum_{(j_1,j_2)\in \ccE} T^{(j_1,j_2)}(\mu_{j_1},\mu_{j_2}),
	\end{equation}
	where $T^{(j_1,j_2)}(\cdot,\cdot)$ is the optimal transport problem \eqref{eq:omt_bi} with cost matrix $C^{(j_1,j_2)}$ for $(j_1,j_2)\in \ccE$.
In particular, both problems have common optimal solutions which are identical in the sense that the projections of an optimal tensor $\bM$ in  \eqref{eq:omt_multi_discrete} coincide with a set of optimal transport plans in \eqref{eq:omt_pairwise_unreg} in the sense that $P_{(j_1,j_2)}(\bM) = M^{(j_1,j_2)}$ for $(j_1,j_2)\in \ccE$.
	
	{\makeatletter
		\let\par\@@par
		\par\parshape0
	\everypar{}
	\begin{wrapfigure}{r}{0.35\textwidth}
		\centering
		\vspace{-15pt}
		\begin{tikzpicture}
		\tikzstyle{main}=[circle, minimum size = 8mm, thick, draw =black!80, node distance = 15mm]
		\node[main,fill=black!10] (mu1) {$\mu_1$};
		\node[main] (mu2) [above right=of mu1]{$\mu_2$};
		\node[main] (mu3) [below right=of mu2] {$\mu_3$};
		
		\draw[  thick] (mu1) -- node[above left] {$M^{(1,2)}$}  (mu2);
		\draw[  thick] (mu2) -- node[above right] {$M^{(2,3)}$}  (mu3);
		\draw[  thick] (mu1) -- node[above] {$M^{(1,3)}$}  (mu3); 
		\end{tikzpicture}
		\vspace{-10pt}
		\caption{Illustration of the cyclic graph in Remark~\ref{rem:cycle}.}
		\vspace{-15pt}
		\label{fig:cycle}
	\end{wrapfigure}
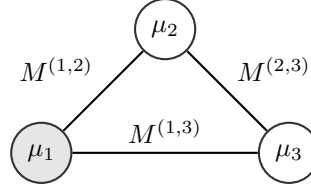
	This, however, can not be generalized to graphs $\ccG=(\ccV,\ccE)$ that contain cycles. A counterexample can be found for a complete graph with three nodes (Figure \ref{fig:cycle}).
		Let the cost matrices be 
		\begin{equation}
		C^{(1,2)} \!=\! C^{(2,3)} \!=\! J \!=\! \begin{bmatrix} 0 & 1 \\ 1 & 0 \end{bmatrix}\!, \quad C^{(1,3)} \!=\! I \!=\!\begin{bmatrix} 1 & 0 \\ 0 & 1 \end{bmatrix}\!,
		\end{equation}
		where $\mu_1= [1,\;1 ]^T$, and $\Gamma=\{1\}$. The unique optimal solution to \eqref{eq:omt_pairwise_unreg} with these parameters are	$\mu_2=\mu_3=\begin{bmatrix} 1 & 1 \end{bmatrix}^T, \quad 	M^{(1,2)}= M^{(2,3)} = I, \quad M^{(1,3)} = J,$	and the objective value is $0$. 
		However, note that the corresponding multimarginal cost tensor given by $\bC_{i_1,i_2,i_3}=C^{(1,2)}_{i_1,i_2}+C^{(1,3)}_{i_1,i_3}+C^{(2,3)}_{i_2,i_3}$ is elementwise strictly positive, and thus the multimarginal problem cannot attain the value $0$. In fact, there is no tensor $\bM\in \mR_+^{2\times 2\times 2}$ that is consistent with these projections $P_{1,2}(\bM)=P_{2,3}(\bM)=I$ and $P_{1,3}(\bM)=J$.
		\par }
	
Here we have illustrated a fundamental difficulty that must be handled when extending tree-graphs to graphs with cycles.
Thus, naively extending the Schr\"odinger bridge from trees to general graphs by defining transition probability matrices $A^{(j_1,j_2)} =  \exp \left( - C^{(j_1,j_2)}/ \epsilon \right)$,
	for all $(j_1,j_2)\in\ccE$, and solving \eqref{eq:HMM_tree} does not yield an equivalence result with a multi-marginal optimal transport problem as in Theorem~\ref{thm:equivalence}.
\end{remark}

\section{Multi-marginal vs. pairwise regularization}
\label{sec:pairwise}

Another natural way to define an optimal transport problem on the tree $\ccT$ is to minimize the sum of all bi-marginal transport costs on the edges of $\ccT$, as in \eqref{eq:omt_pairwise_unreg}.
	In fact, the unregularized multi-marginal optimal transport problem \eqref{eq:omt_multi_discrete} structured according to $\ccT=(\ccV,\ccE)$ is equivalent to this pairwise problem.	
	An alternative computational approach to the one taken in this paper is thus to regularize each term in the sum of \eqref{eq:omt_pairwise_unreg} by an entropy term.
The entropy regularized pairwise optimal transport problem on $\ccT$ is then  
\begin{equation} \label{eq:omt_pairwise}
\minwrt[  \mu_{j}, j \in \ccV \setminus \Gamma] \  \sum_{(j_1,j_2)\in \ccE} T_\epsilon^{(j_1,j_2)}(\mu_{j_1},\mu_{j_2}),
\end{equation}
where $T^{(j_1,j_2)}_\epsilon(\cdot,\cdot)$ is defined as the regularized bi-marginal optimal transport problem \eqref{eq:omt_reg} with cost matrix $C^{(j_1,j_2)}$.
In particular, it is common to formulate barycenter problems in this pairwise regularized manner \cite{benamou2015bregman, cuturi2014barycenter, solomon2015,  kroshnin2019complexity, lin2020revisiting,  bonneel2016wasserstein}.
Recently, we have empirically observed that in some applications the barycenter problem with multi-marginal regularization as in \eqref{eq:omt_multi_regularized} gives better results as compared to the pairwise regularization in \eqref{eq:omt_pairwise}, see \cite[Sec. 6.3]{elvander19multi}. Therein solutions using the multi-marginal regularization are less smoothed out as compared to the pairwise regularization. In the following we investigate the difference between the two problems. In particular, we show that the latter is not equivalent to the generalized Schr\"odinger bridge \eqref{eq:HMM_tree}.

Consider a rooted version of the tree, denoted $\ccT_r=(\ccV,\ccE_r)$, where $r\in\ccL$.
In the case that $A^{(j_1,j_2)} = K^{(j_1,j_2)}$, for all $(j_1,j_2)\in\ccE_r$, the multi-marginal optimal transport problem has the same solution as problem \eqref{eq:HMM_tree_anyroot}. Now note that using Proposition~\ref{prp:HMM_tree_anyroot} this problem can be written as
\begin{equation} \label{eq:omt_pairwise_hmm}
\minwrt[  \mu_{j}, j \in \ccV \setminus \Gamma] \! \!  \sum_{(j_1,j_2)\in \ccE} \! \!  T_\epsilon^{(j_1,j_2)}(\mu_{j_1},\mu_{j_2}) - \sum_{j\in\ccV \setminus \ccL} (\deg(j)-1)H(\mu_j).
\end{equation}
Thus, the generalized Schr\"odinger bridge problem \eqref{eq:HMM_tree}, or equivalently \eqref{eq:omt_pairwise_hmm}, is not equivalent to the pairwise regularized optimal transport problem \eqref{eq:omt_pairwise}, unless for the trivial case of a tree with only 2 vertices.
Moreover, we can see from \eqref{eq:omt_pairwise_hmm} that the multi-marginal optimal transport problem penalizes not only the transport cost between the marginals, but in addition favors marginal distributions with low entropy.%
\footnote{Recall that the definition of $H(\mu_j)$ essentially corresponds to the negative of the entropy of $\mu_j$.}
One can thus expect less smoothed out distributions when solving the multi-marginal optimal transport problem, which is desirable in many applications, such as localization problems \cite{elvander19multi} and computer vision applications \cite{solomon2015}.
Solving tree-structured problems by using multi-marginal regularization thus has some advantages as compared to pairwise regularization. Firstly, it yields less smoothed out solutions, and secondly it preserves the connections to the Schr\"odinger bridge problem.
Finally,
empirical study suggests that the multi-marginal problem is better conditioned compared to the pairwise problem, which allows for smaller values of the regularization parameter $\epsilon$, while still yielding a numerically stable algorithm. We give empirical evidence for this behaviour in Section~\ref{sec:ex_tree}.

The solution of the pairwise optimal transport problem \eqref{eq:omt_pairwise} is compared to the solution of the two equivalent problems \eqref{eq:omt_multi_regularized} and \eqref{eq:HMM_tree} on the example of a path graph in Section~\ref{sec:ex_bridge}, and on a more complex tree in Section~\ref{sec:ex_tree}.

\subsection{The discrete time Schr\"odinger bridge problem} \label{sec:ex_bridge}

%
	%
	\begin{figure}
		\centering
		\begin{tikzpicture}
		\tikzstyle{main}=[circle, minimum size = 8mm, thick, draw =black!80, node distance = 15mm]
		\node[main,fill=black!10] (mu1) {$\mu_1$};
		\node[main] (mu2) [right=of mu1]{$\mu_2$};
		\node[] (mu3) [right=of mu2] {};
		\node[] (muJm1) [right=of mu3] {};  
		\node[main,fill=black!10] (muJ) [right=of muJm1] {$\mu_J$};
		
		\draw[->, -latex, thick] (mu1) -- node[above] {$M^{(1,2)}$} node[below] {$A_1$} (mu2);
		\draw[->, -latex, thick] (mu2) -- node[above] {$M^{(2,3)}$} node[below] {$A_2$} (mu3);
    	\draw[loosely dotted, very thick] (mu3) -- (muJm1); 
		\draw[->, -latex, thick] (muJm1) -- node[above] {$M^{(J-1,J)}$} node[below] {$A_{J-1}$} (muJ); 
		\end{tikzpicture}
		\caption{Illustration of the linear path tree in Section~\ref{sec:ex_bridge}.}  \vspace{-10pt}
		\label{fig:model_bridge}
	\end{figure}
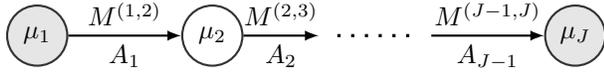
%
Consider a path tree $\ccT_r=(\ccV,\ccE_r)$, where $\ccV=\{1,2,\dots,J\}$ and $\ccE_r= \{ (j,j+1) | j=1,\dots,J-1\}$, as sketched in Figure~\ref{fig:model_bridge}. Let $A^j$ denote the probability transition matrix on $(j,j+1)\in\ccE_r$. This model corresponds to a Markov chain of length $J$. Assume that the distributions on the leaves $j=1$ and $j=J$ are known. The most likely particle evolutions between them are then found by solving \eqref{eq:opt_markov_chain}.
Following \cite{haasler19ensemble}, we see that this problem is equivalent to the discrete time and discrete space Schr\"odinger bridge in \cite{pavon2010discrete}.
Assume that the distributions $\mu_j$, for $j=1,\dots,J$, are strictly positive, and define the row stochastic matrices $\bar A^j = \diag(\mu_j)^{-1} M^j$, for $j=1,\dots,J$.
In terms of these matrices problem \eqref{eq:opt_markov_chain} reads
\begin{equation} \label{eq:discrete_schrodinger_bridge}
\begin{aligned}
\minwrt[\bar A_{[1:J-1]}, \mu_{[2:J-1]}] \  & \sum_{j=1}^{J-1} \sum_{i=1}^n (\mu_j)_i H \left(  \bar A^j_{i \cdot} \ | \ A^j_{i \cdot}\right)  \\
\text{subject to } \ & \bar A^j \mathbf{1} = \ett ,\ \  \mu_{j+1} = (\bar A^j)^T \mu_{j},\ \ \text{for } \  j=1,\dots,J-1.
\end{aligned}
\end{equation}
Here $A_{i \cdot}$ denotes the $i$-th row of $A$. Problem \eqref{eq:discrete_schrodinger_bridge} is exactly the formulation of a Schr\"odinger bridge over a Markov chain from \cite[eq.~(24)]{pavon2010discrete}.
In \cite{pavon2010discrete} it is shown that a unique solution to a corresponding Schr\"odinger system exists if $\mu_J$ is a strictly positive distribution and the matrix $ \prod_{j=1}^{J-1} A^j$ has only positive elements. The solution to the Schr\"odinger system may be obtained as a fixed point iteration \cite{Georgiou2015discreteSB}, which is linked to the Sinkhorn iterations for entropy regularized optimal transport problems. 
We recall from \cite{haasler19ensemble} that the optimization problem \eqref{eq:discrete_schrodinger_bridge} is non-convex, whereas the equivalent formulation \eqref{eq:opt_markov_chain} is convex.

We note that with cost matrices defined as $C^j= -\epsilon \log(A^j)$, for $j=1,\dots,J-1$, the discrete Schr\"odinger bridge problem may be written as
\begin{equation}
\minwrt[  \mu_{2},\dots, \mu_{J-1}]   \sum_{j=1}^{J-1}  T_\epsilon^{(j,j+1)}(\mu_{j},\mu_{j+1}) - \sum_{j=2}^{J-1} H(\mu_j).
\end{equation}
In comparison to the pairwise optimal transport problem on the path graph $\ccT$, the Schr\"odinger bridge thus favors intermediate distributions with lower entropy, resulting in less smoothed out solutions.
\begin{figure} 
	\centering
	\subfigure[$\epsilon=10^{-2}$]{ \includegraphics[trim={13pt 6pt 33pt 32pt},clip,width=0.225\columnwidth]{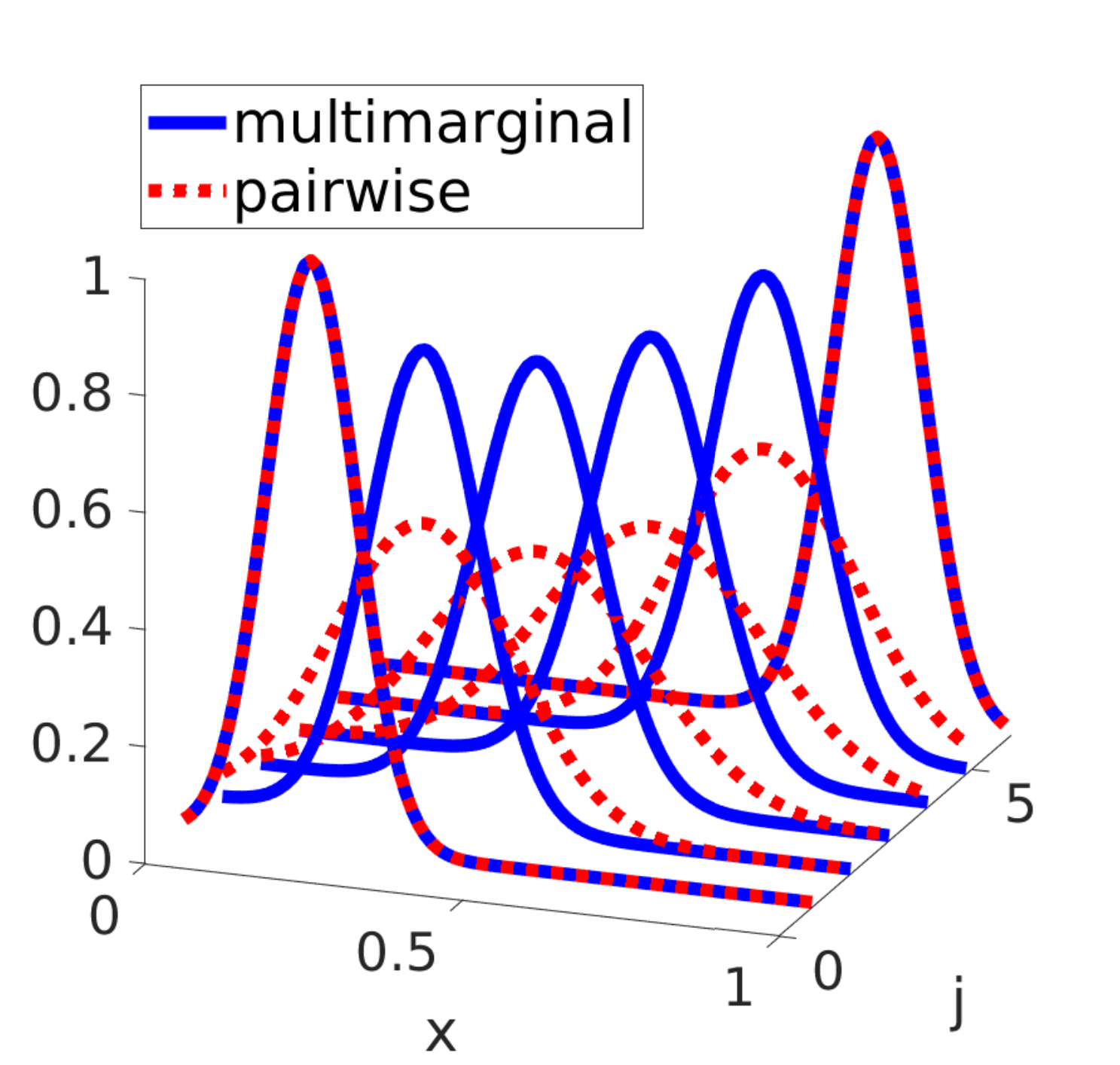} }
	\subfigure[$\epsilon=5 \cdot 10^{-3}$]{ \includegraphics[trim={13pt 6pt 33pt 32pt},clip,width=0.225\columnwidth]{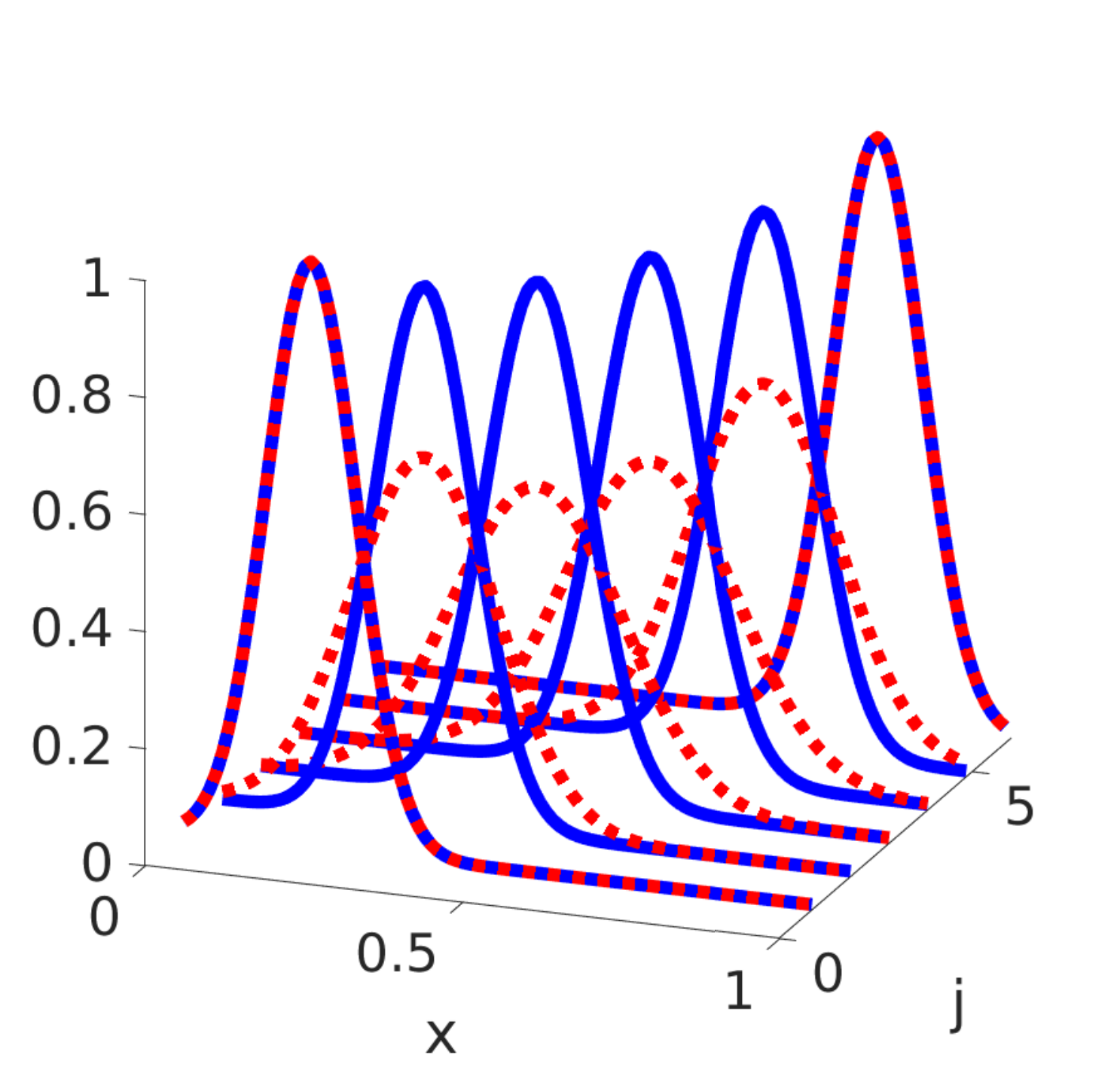} }
	\subfigure[$\epsilon=10^{-3}$]{ \includegraphics[trim={13pt 6pt 33pt 32pt},clip,width=0.225\columnwidth]{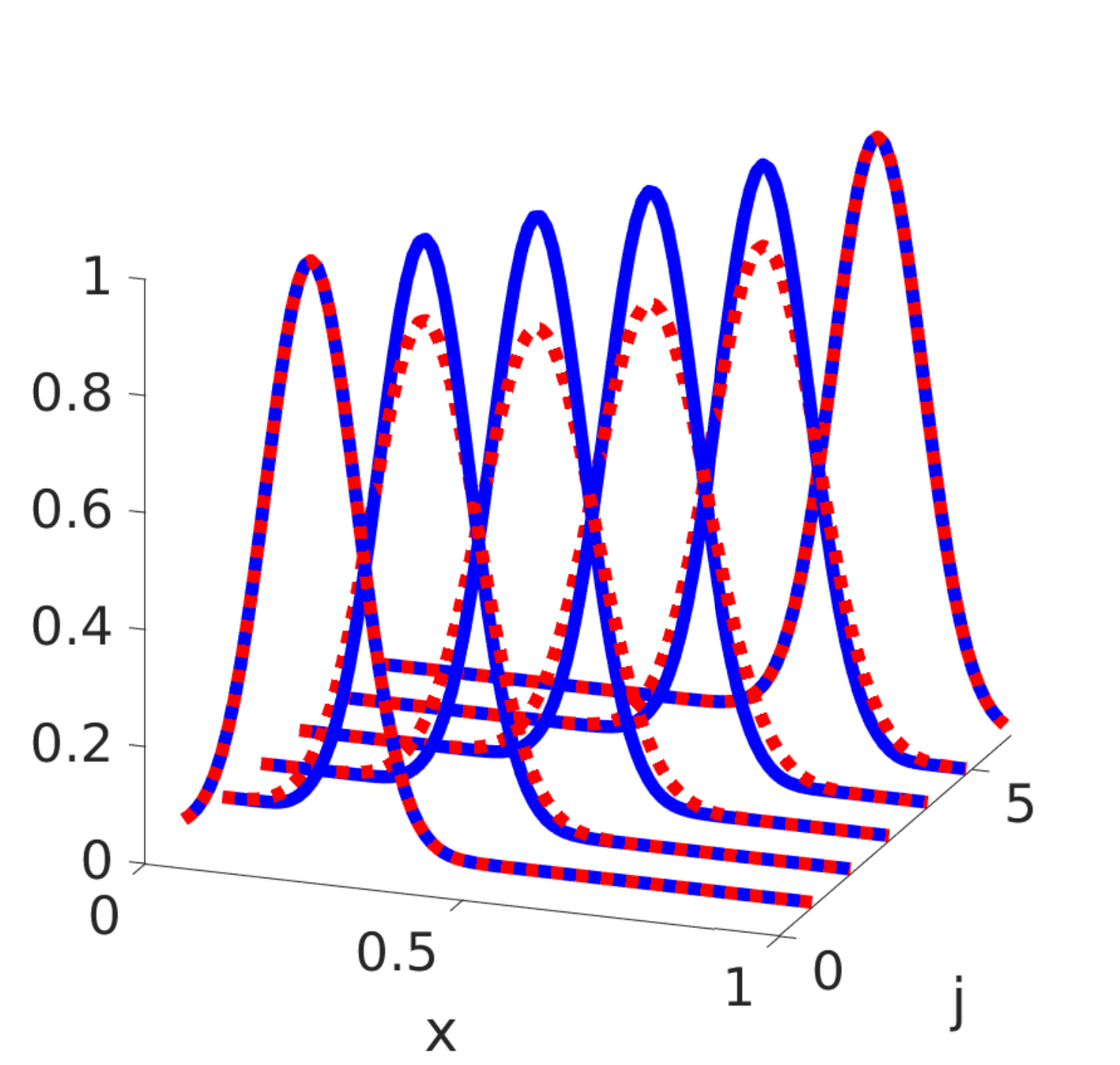} }
	\subfigure[$\epsilon=5 \cdot 10^{-4}$]{ \includegraphics[trim={13pt 6pt 33pt 32pt},clip,width=0.225\columnwidth]{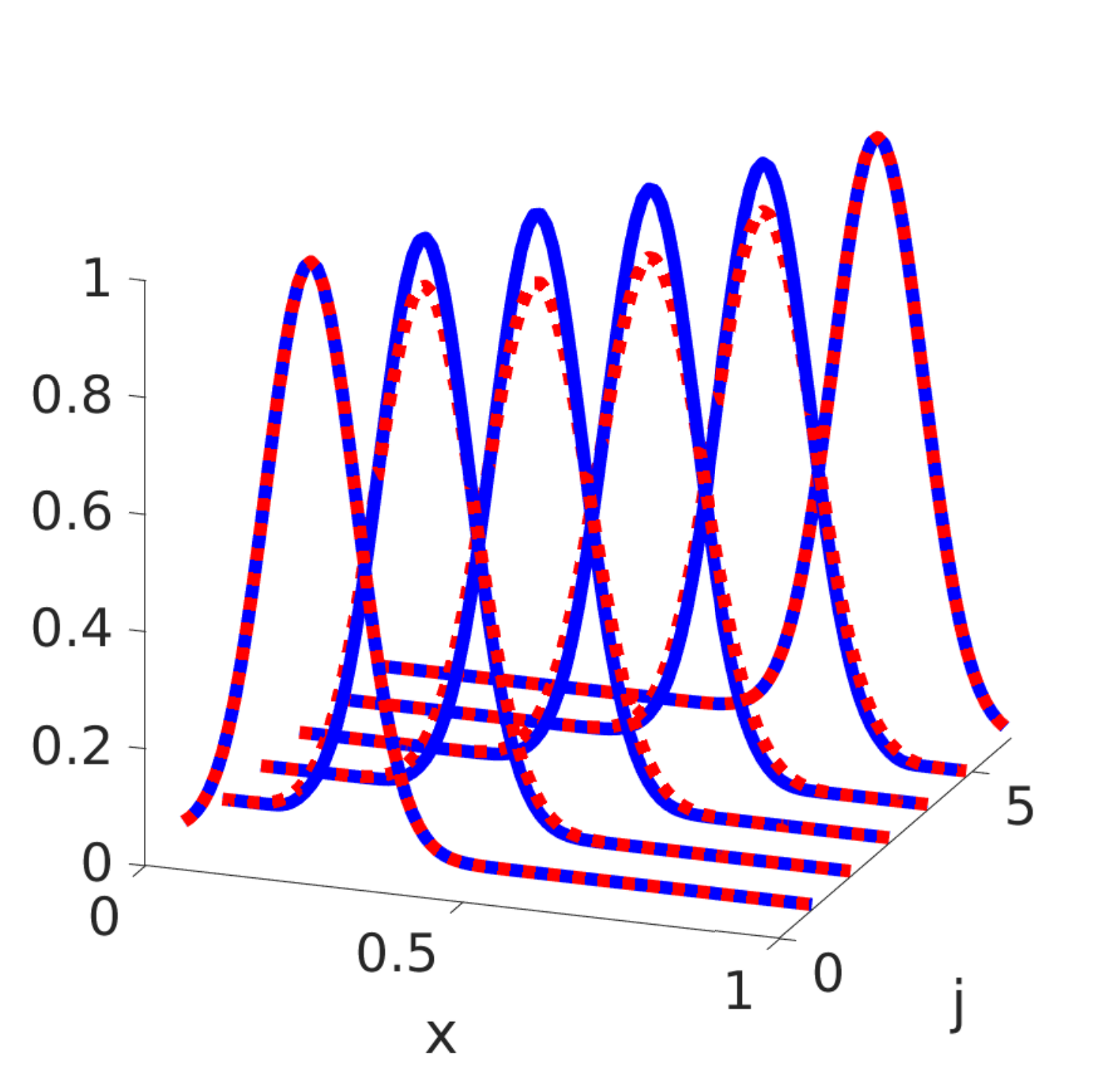} }
	\caption{Solutions to the multi-marginal and pairwise entropy regularized optimal transport problem on a path graph for varying regularization parameter $\epsilon$.} \vspace{-10pt}
	\label{fig:bridge_solution}
\end{figure}
We illustrate this behaviour for a path tree $\ccT_r=(\ccV,\ccE_r)$, with vertices $\ccV=\{1,2,\dots,6\}$, and where the initial and final distribution are given as $\mu_1(x)= \exp(-(\frac{x-0.2}{10})^2)$, and $\mu_J(x)= \exp( - (\frac{x-0.8}{10})^2)$, and the cost function is defined as the Euclidean distance.
The results to problem~\eqref{eq:discrete_schrodinger_bridge} and problem~\eqref{eq:omt_pairwise} are compared for different values of the regularization parameter $\epsilon$ in Figure~\ref{fig:bridge_solution}.
Both optimal transport solutions describe a smooth way of shifting the mass from distribution $\mu_1$ to $\mu_J$. The larger the regularization parameter $\epsilon$ is chosen, the more smoothed out are the intermediate solutions to both problems. However, for each value of $\epsilon$ the multi-marginal optimal transport solution infers substantially less smoothing than the pairwise optimal transport solution.

\subsection{Optimal transport with tree-structure} \label{sec:ex_tree}
In this section we compare the entropy regularized multi-marginal and pairwise optimal transport solutions on a more general tree.
Consider the tree $\ccT=(\ccV,\ccE)$ illustrated in Figure~\ref{fig:tree_a} with 15 nodes, each representing a $50\times50$ pixel image. The marginal images on the 8 leaves, coloured in gray, are known.
Each edge on the tree $\ccT$ is associated with a cost function defined by the Euclidean distance between any two pixels. Using this choice of cost function in the corresponding optimal transport problems yields smooth translations in power for the intermediate marginals.

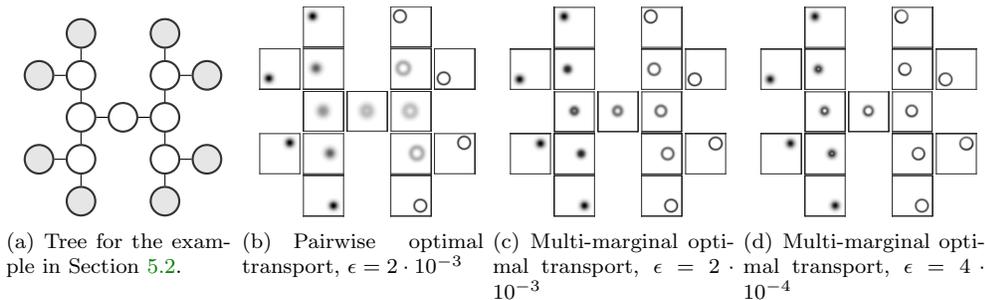
\begin{figure}
	\centering
	\subfigure[Tree for the example in Section~\ref{sec:ex_tree}.]{ \label{fig:tree_a}  \centering
		\begin{tikzpicture}
		\tikzstyle{main}=[circle, minimum size = 3.8mm, thick, draw =black!80, node distance = 1.5mm]
		\node[main,fill=black!10] (mu1) {};
		\node[main] (mu3) [below=of mu1] {};
		\node[main,fill=black!10] (mu2) [left=of mu3] {};
		\node[main] (mu4) [below=of mu3] {};
		\node[main] (mu5) [below=of mu4] {};
		\node[main,fill=black!10] (mu6) [left=of mu5] {};
		\node[main,fill=black!10] (mu7) [below=of mu5] {};
		\node[main] (mu8) [right=of mu4] {};
		\node[main] (mu9) [right=of mu8] {};
		\node[main] (mu10) [above=of mu9] {};
		\node[main,fill=black!10] (mu11) [above=of mu10] {};
		\node[main,fill=black!10] (mu12) [right=of mu10] {};
		\node[main] (mu13) [below=of mu9] {};
		\node[main,fill=black!10] (mu14) [below=of mu13] {};
		\node[main,fill=black!10] (mu15) [right=of mu13] {};

		\draw (mu1) --(mu3);
		\draw (mu2) --(mu3);
		\draw (mu3) --(mu4);
		\draw (mu4) --(mu5);
		\draw (mu5) --(mu6);
		\draw (mu5) --(mu7);
		\draw (mu4) --(mu8);
		\draw (mu8) --(mu9);
		\draw (mu9) --(mu10);
		\draw (mu10) --(mu11);
		\draw (mu10) --(mu12);
		\draw (mu9) --(mu13);
		\draw (mu13) --(mu14);
		\draw (mu13) --(mu15);
		\end{tikzpicture}
	}
	\hfill
	\subfigure[Pairwise optimal transport, $\epsilon=2 \cdot 10^{-3}$ ]{ \label{fig:tree_b} \includegraphics[trim={4pt 2pt 4pt 3pt}, clip,  width=0.22\columnwidth] {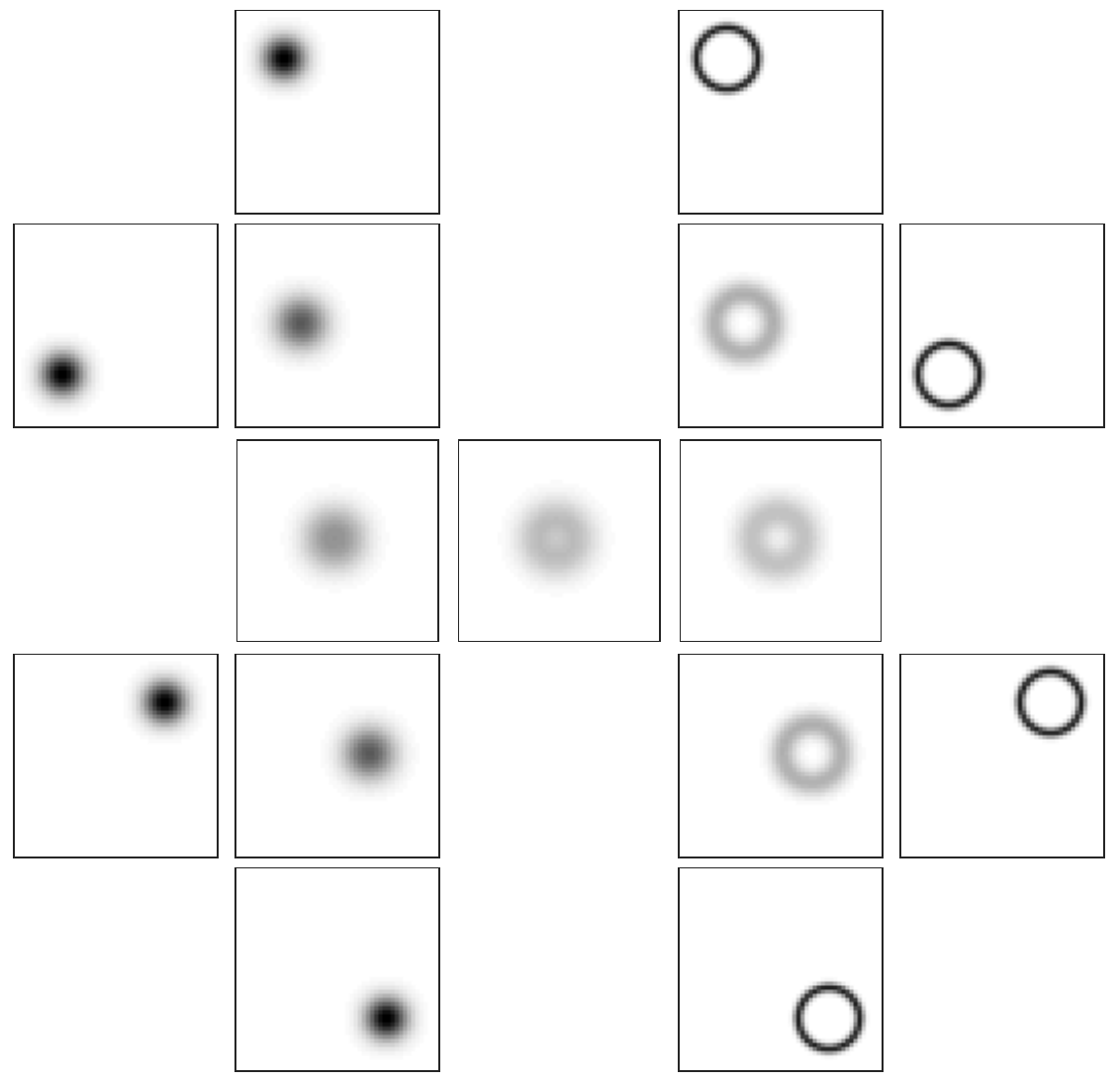}  }
	\subfigure[Multi-marginal optimal transport, $\epsilon=2 \cdot 10^{-3}$]{ \label{fig:tree_c}  \includegraphics[trim={4pt 2pt 4pt 3pt}, clip, width=0.22\columnwidth]{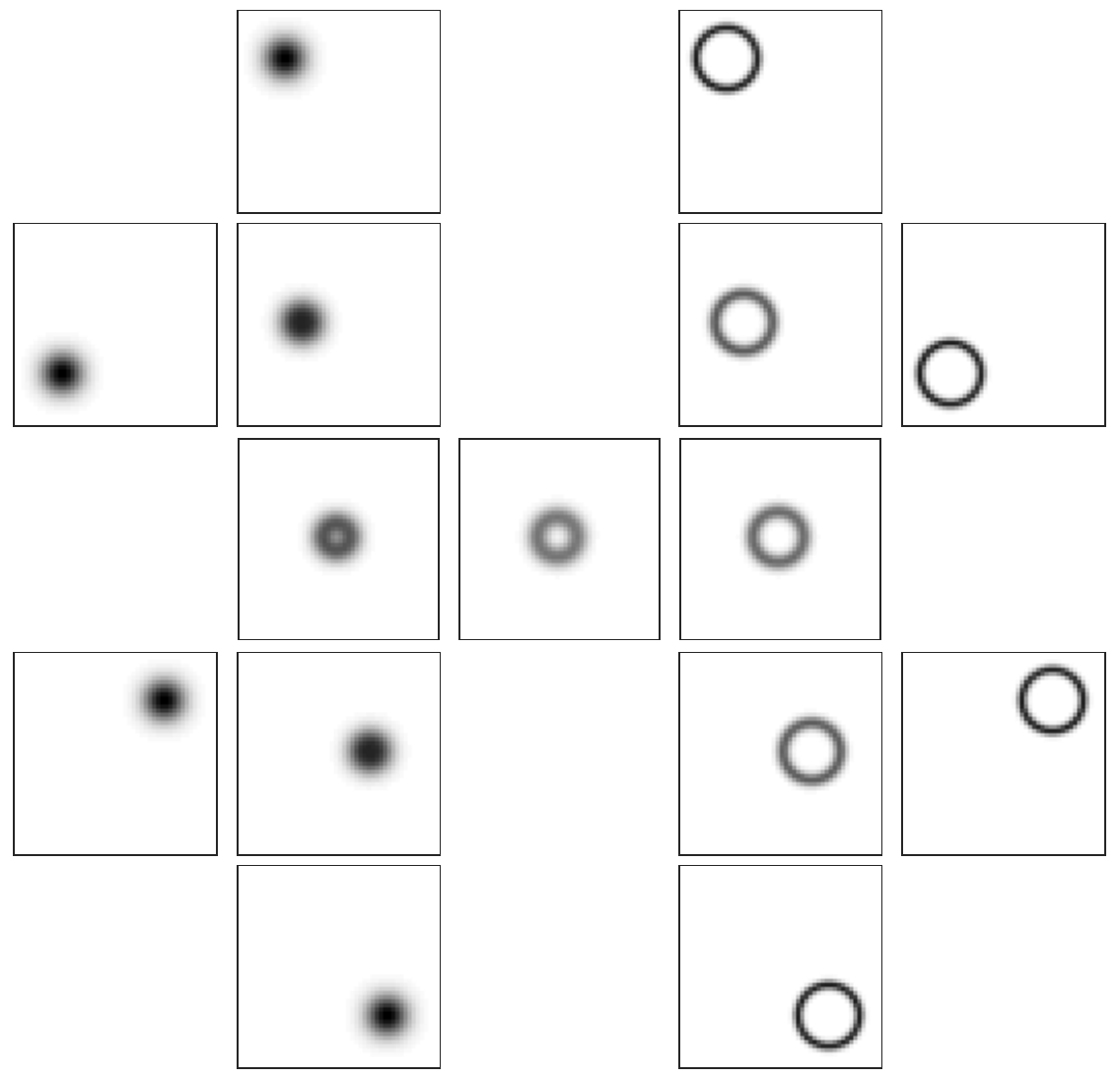} }
	\subfigure[Multi-marginal optimal transport, $\epsilon=4 \cdot 10^{-4}$]{ \label{fig:tree_d}  \includegraphics[trim={4pt 2pt 4pt 3pt}, clip, width=0.22\columnwidth]{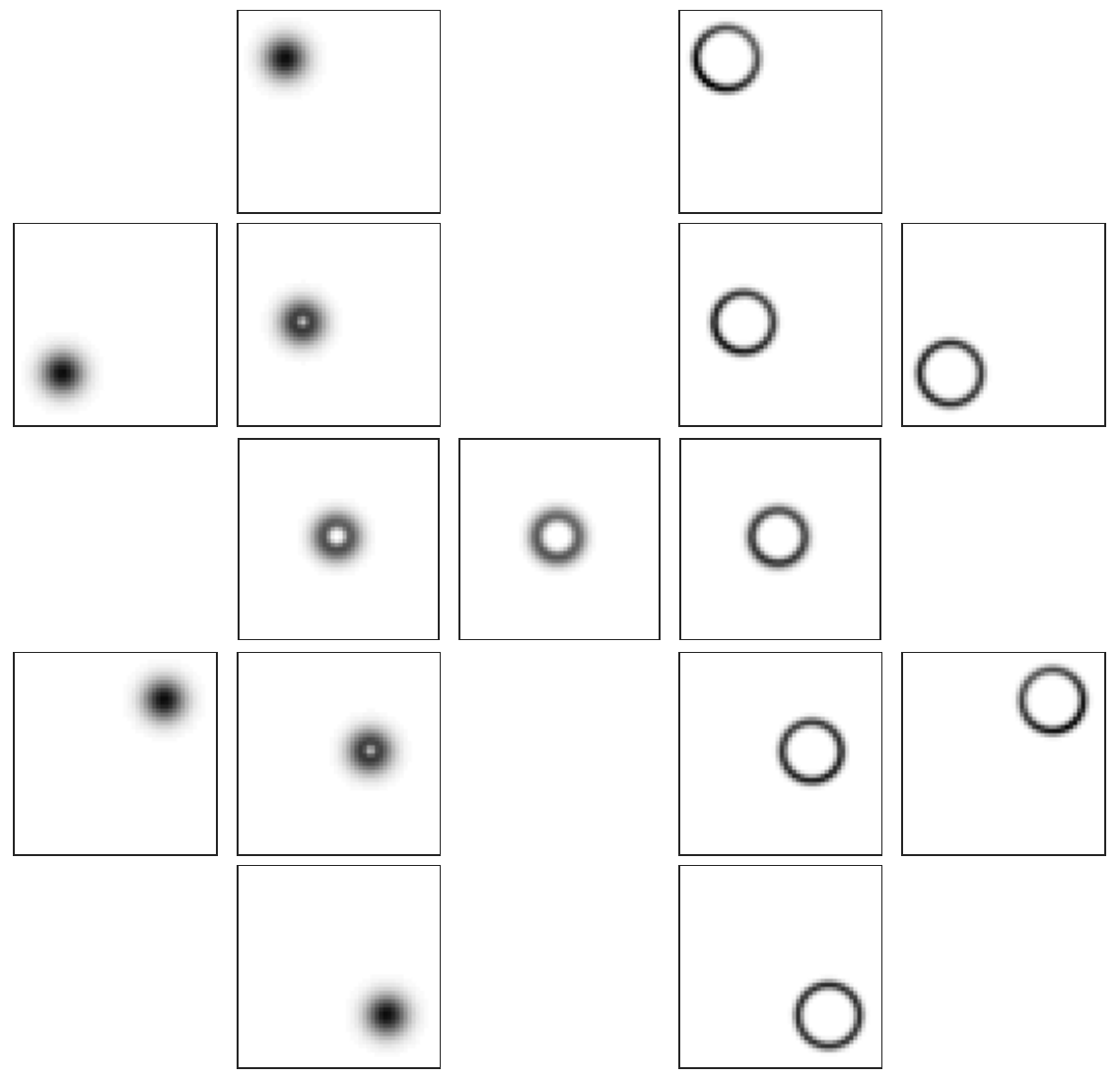} }
	\caption{Estimated marginals of the pairwise (b) and multi-marginal (c,d) optimal transport solutions on the tree in (a).}  \vspace{-10pt}
	\label{fig:omt_tree}
\end{figure}

We solve the entropy regularized pairwise optimal transport problem \eqref{eq:omt_pairwise} with regularization parameter $\epsilon=2 \cdot 10^{-3}$ on $\ccT$. The solution can be seen in Figure~\ref{fig:tree_b}. Compared to the pairwise optimal transport estimate, the solution to the entropy regularized multi-marginal optimal transport problem \eqref{eq:omt_multi_regularized} on the same tree $\ccT$ and with the same regularization parameter $\epsilon$ is significantly sharper and less smoothed out, see Figure \ref{fig:tree_c}.
For the pairwise optimal transport problem, the method diverges with a smaller regularization parameter, 
e.g., $\epsilon=10^{-3}$.
In contrast, for the multi-marginal formulation the regularization parameter can be decreased further, still yielding a numerically stable algorithm. We have found that the method is still stable for a regularization parameter of $\epsilon=4\cdot 10^{-4}$, which results in very clear estimates on the intermediate nodes.

\section{Estimating ensemble flows on a hidden Markov chain} \label{sec:ensemble_flows}
We consider the problem of tracking an ensemble of agents on a network based on aggregate measurements from sensors distributed around the network. 
This is similar to \cite{haasler19ensemble}, where ensemble flows of indistinguishable agents have been estimated as the maximum likelihood solution on a hidden Markov chain.
The present work generalizes this method and provides an algorithm for solving the problem introduced therein.
In particular, the framework in \cite{haasler19ensemble} is a special case of the method in Theorem~\ref{thm:HMM_tree_sol}, and can therefore be solved with Algorithm~\ref{alg:sinkhorn}.
Herein, we study different observation models and robustness of the estimates with respect to the number of agents.
\begin{figure}
\subfigure[Illustration of the Markov model.]{ \label{fig:HMM_tree_model}
	\begin{tikzpicture}
	\tiny
	\tikzstyle{main}=[circle, minimum size =8mm, thick, draw =black!80, node distance = 6mm]
	\tikzstyle{obs}=[circle, minimum size = 8mm, thick, draw =black!80, node distance = 4 mm and 2.5mm ]
	\node[main] (mu1) {$\mu_1$};
	\node[main] (mu2) [right=of mu1] {$\mu_2$};
	\node[] (mu3) [right=of mu2] {};
	\node[] (muTm1) [right=of mu3] {};  
	\node[main] (muT) [right=of muTm1] {$\mu_\tau$};
	
	\node[] (phi1c) [below=of mu1] {};  
	\node[obs,fill=black!10] (phi11) [left=of phi1c] {$\Phi_{1,1}$};  
	\node[obs,fill=black!10] (phi1S) [right=of phi11] {$\Phi_{1,S}$};
	\node[obs,fill=black!10] (phi21) [right=of phi1S] {$\Phi_{2,1}$};
	\node[obs,fill=black!10] (phi2S) [right=of phi21] {$\Phi_{2,S}$};
	\node[node distance = 3mm] (phi3) [right=of phi2S] {};
	\node[node distance = 3mm] (phiTm1) [right=of phi3] {};
	\node[obs,fill=black!10] (phiT1) [right=of phiTm1] {$\Phi_{\tau,1}$};
	\node[obs,fill=black!10] (phiTS) [right=of phiT1] {$\Phi_{\tau,S}$};
	
	\node[] (space) [below=of phi1c] {};
	
	\draw[->, -latex, thick] (mu1) -- node[above] {$M^1$} (mu2);
	\draw[->, -latex, thick] (mu2) -- node[above] {$M^2$} (mu3); 
	\draw[loosely dotted, very thick] (mu3) -- (muTm1); 
	\draw[->, -latex, thick] (muTm1) -- node[above] {$M^{\tau-1}$} (muT); 
	
	\draw[->, -latex, thick] (mu1) -- node[left] {$D^{1,1}$} (phi11);
	\draw[->, -latex, thick] (mu1) -- node[right] {$D^{1,S}$} (phi1S);
	\draw[->, -latex, thick] (mu2) -- node[left] {$D^{2,1}$} (phi21);
	\draw[->, -latex, thick] (mu2) -- node[right] {$D^{2,S}$} (phi2S);
	\draw[->, -latex, thick] (muT) -- node[left] {$D^{\tau,1}$} (phiT1);
	\draw[->, -latex, thick] (muT) -- node[right] {$D^{\tau,S}$} (phiTS);
	
	\draw[loosely dotted, very thick] (phi11) -- (phi1S); 
	\draw[loosely dotted, very thick] (phi21) -- (phi2S); 
	\draw[loosely dotted, very thick] (phiT1) -- (phiTS); 
	\draw[loosely dotted, very thick] (phi3) -- (phiTm1); 
	\end{tikzpicture}
	} 
	\subfigure[Network and sensors.]{ \label{fig:ensemble_network}
		\includegraphics[trim={40pt 20pt 40pt 10pt}, clip, width=0.37\textwidth]{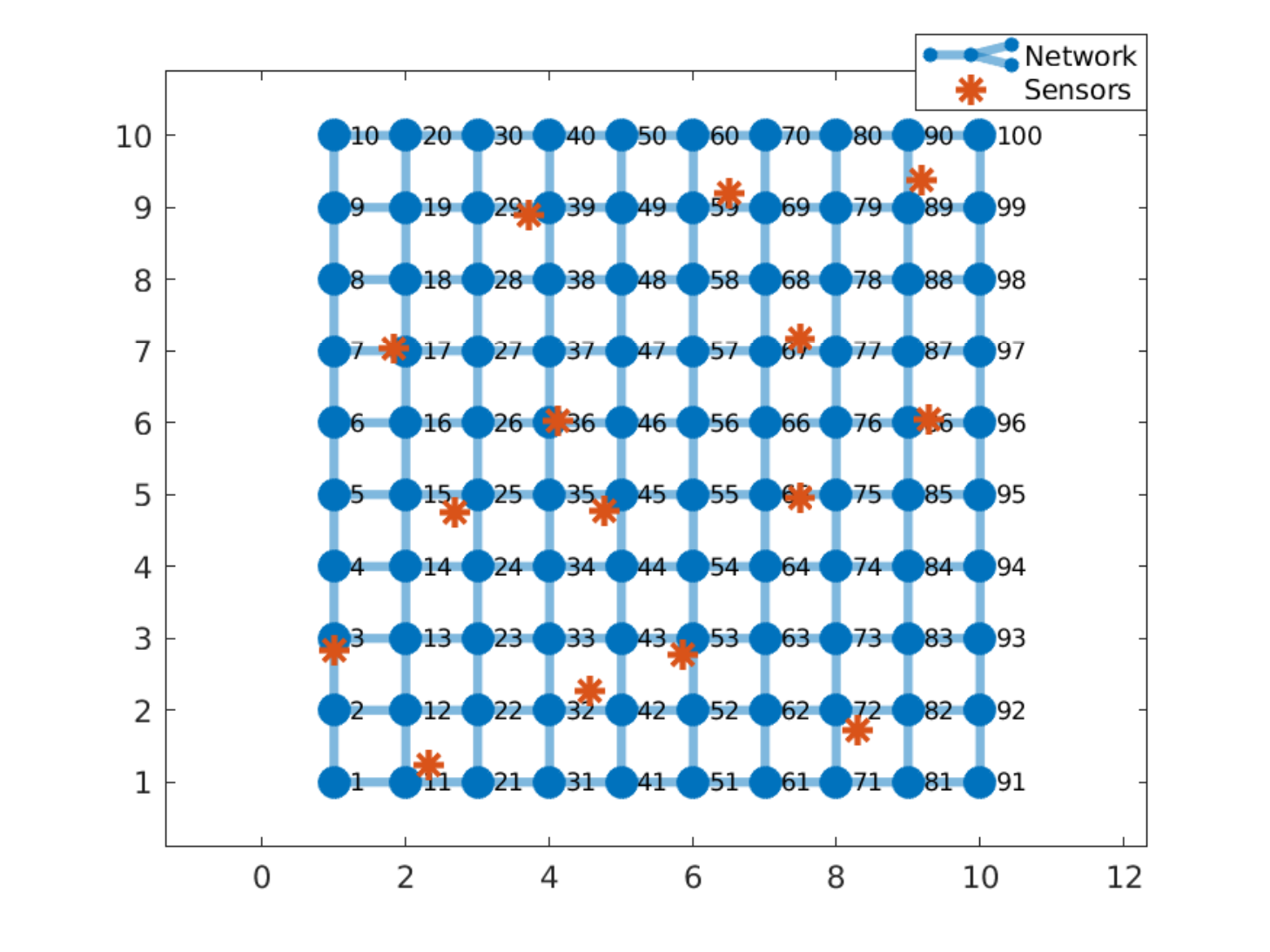}	} \vspace{-5pt}
	\caption{Ensemble flow estimation example in Section~\ref{sec:ensemble_flows}.} \vspace{-5pt}
\end{figure}
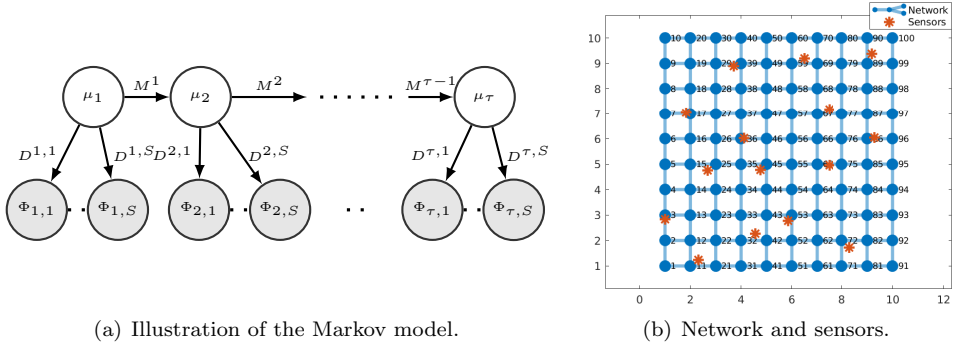

Consider a graph $\ccG=(\ccV_\ccG, \ccE_\ccG)$ with vertices $\ccV_\ccG$ and edges $\ccE_\ccG$. 
Let the set of nodes $\ccV_\ccG$ be the set of states of a Markov model with transition probability matrix $A^t\in \mR^{n\times n}$, where $n=|\ccV_\ccG|$, for $t=1,\dots,\tau-1$.
We simulate a finite number of agents to evolve according to this Markov model for a number of time steps $t=1,\dots,\tau$.
Let $\mu_t\in \mR^{n}$ denote the distribution of agents over the set of states for the times $t=1,\ldots, \tau$.
An observation model is represented by the detection probability matrices $B^s\in \mR^{n\times m}$ where $m$ is the size of the observation space, and $s=1,\ldots, S$ where $S$ denotes the number of uncoupled observations at each time point. 
Let $\Phi_{t,s}\in \mR^m$ denote the aggregate measurements at  time $t$ for observation $s$. 
We then estimate the ensemble evolution as the solution of
\begin{equation}
\begin{aligned}
 \minwrt[M_{[1:\tau-1]}, D_{[1:\tau],[1:S]}, \mu_{[1:\tau]}] \ & \sum_{t=1}^{\tau-1} H(  M^t \ | \ \diag(\mu_{t})A^t)  + \sum_{t=1}^{\tau} \sum_{s=1}^S H( D^{s,t} \ | \ \diag(\mu_t) B^s ) \nonumber\\
\text{subject to} \qquad \quad &  M^t \ett = \mu_{t-1} ,\;\; (M^t)^T \ett = \mu_{t}, \;\; D^{t,s} \ett = \mu_t, \;\; (D^{t,s})^T \ett = \Phi_{t,s}, \label{eq:KL_multi_measurement} \\
 &  \nonumber  \text{for } t=1,\dots, \tau, \text{ and } s=1,\dots,S. \nonumber
\end{aligned}
\end{equation}
The tree corresponding to the Markov model \eqref{eq:KL_multi_measurement} is sketched in Figure~\ref{fig:HMM_tree_model}
This optimization problem is similar to the one in \cite{haasler19ensemble}, but therein the initial distribution of agents, i.e., the marginal $\mu_1$, is assumed to be known.

In the following we consider a number of $N$ agents evolving on the network displayed in Figure~\ref{fig:ensemble_network} with $n=100$ nodes and 180 edges. Figure~\ref{fig:ensemble_network} also shows the location of the $N_S = 15$ sensors.
In a first step, we compute the discrete Schr\"odinger bridge \eqref{eq:opt_markov_chain} between the distribution $\mu_1$, where all agents are in node $1$, and the distribution $\mu_{\tau}$, where all agents are in node $100$, with the random walk on $\ccG$ as a prior. The resulting particle evolutions $M^t$, for $t=1,\dots,\tau-1$, in \eqref{eq:opt_markov_chain} define the transition probability matrices $\bar A^t=\diag(\mu_t)^{-1}M^t$, for $t=1,\dots,\tau-1$, as in \eqref{eq:discrete_schrodinger_bridge}, used to simulate a number of $N$ agents with initial distribution $\mu_1$. By construction, the final distribution is then $\mu_\tau$.

We consider two observation models, one where the sensors make uncoupled measurements, similar to the example in \cite{haasler19ensemble}, and one where the sensors are coupled and form a joint measurement.
In the uncoupled setting, we consider an observation space of $2$ states for each sensor, where one state denotes that an agent is detected, and the other one that the agent is undetected by the sensor. Hence, we define an observation probability matrix $B^s \in \RR^{n \times 2}$ for each sensor $s=1,\dots,S=N_S$, where the probability for an agent on node $i$ to be detected by the sensor $s$ is defined as $B^s_{i1} = \min(0.99, 2 e^{-d_{s,i}} )$, where $d_{s,i}$ denotes the Euclidean distance between the location of sensor $s$ and the node $i$. Consequently the probability of not being detected is $B^s_{i2}=1-B^s_{i1}$.
For the coupled observation model, the observation space consists of all possible sets of sensors that detect a given agent, i.e., all subsets of the set $\{1,\dots,N_S\}$. The size of the observation space is thus $2^{N_S}$, and there is only one observation at each time instance, i.e., $S=1$. Hence, we define an observation probability matrix $B^{\rm joint} \in \RR^{n \times 2^S}$, where the probability for an agent in node $i$ to be detected by exactly the set $\mathfrak{S}$ of sensors is given by
$ (\prod_{s \in \mathfrak{S}} B^s_{i1})(\prod_{s \notin \mathfrak{S}} B^s_{i2})$.

We solve the multi-marginal optimal transport problem corresponding to \eqref{eq:KL_multi_measurement} with assumed probability transition matrix $A$ describing a random walk on $\ccG$, and observation probability matrices $B^s$ for the two described observation models.
The results for both observation models are compared for $N\in\{10,100,1000\}$  agents in 
Figure \ref{fig:ensemble},
where the estimated number of agents in each node is plotted as a circle with size corresponding to the log-scaled weigthed number of agents in that node.
In Figure~\ref{fig:ensemble_N10}, one can see that for $N=10$ agents the coupled estimate localizes the position of the agents slightly better than the uncoupled estimate.  However, this effect decreases with an increasing number of agents, as the evolution of a single agent has less influence on the empirical distribution of agents. Already from a number of $N=100$ agents the uncoupled estimate
is competitive with the coupled estimate, albeit relying on significantly less information (cf. Figure~\ref{fig:ensemble_N100}). For an ensemble of $N=1000$ agents one can hardly see any difference between the two estimates, as shown in Figure~\ref{fig:ensemble_N1000}.
\begin{figure*}
	\subfigure[$N=10$ agents.]{ \label{fig:ensemble_N10} \includegraphics[ trim={2pt 5pt 12pt 17pt}, clip, width=.3\textwidth]{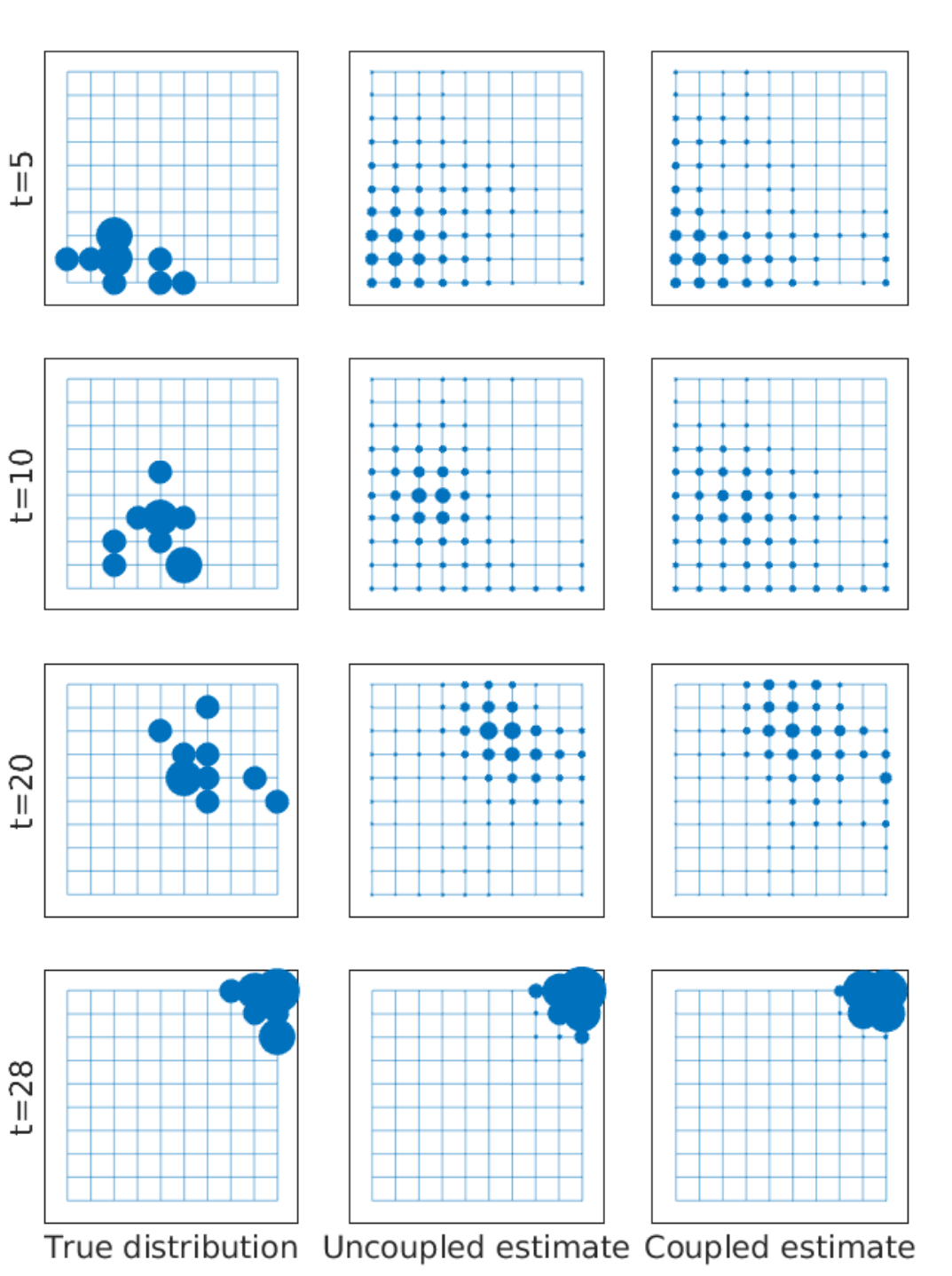}	}
\subfigure[$N=100$ agents.]{ \label{fig:ensemble_N100} \includegraphics[ trim={2pt 5pt 12pt 17pt}, clip, width=.3\textwidth]{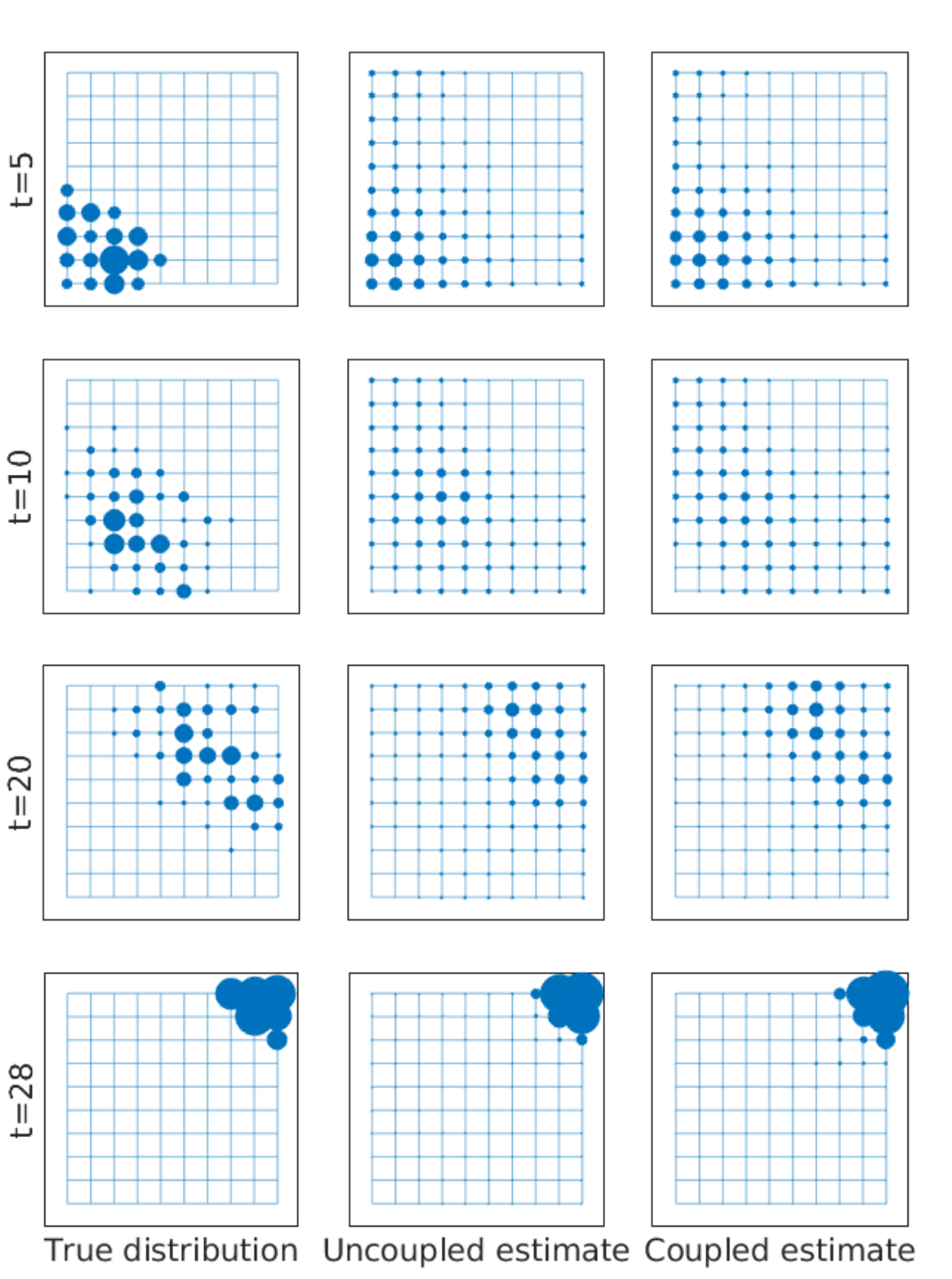}
}
\subfigure[$N=1000$ agents.]{ \label{fig:ensemble_N1000} \includegraphics[ trim={2pt 5pt 12pt 17pt}, clip, width=.3\textwidth]{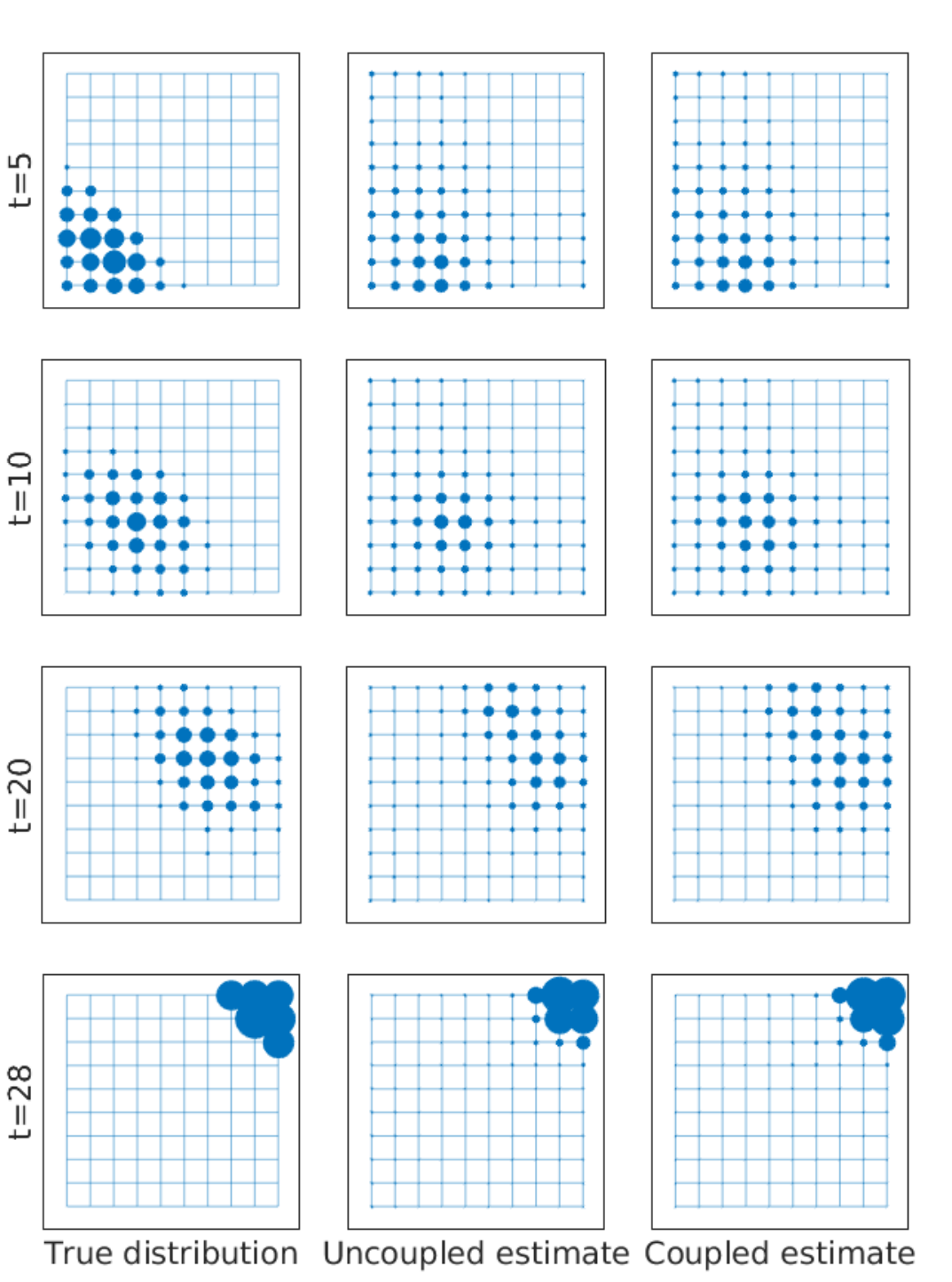}
}
	\caption{True ensemble flow and estimates with the two observation models for a varying number of agents.}  \vspace{-10pt} 
	\label{fig:ensemble}
\end{figure*}

Note that the number of observations at each time instance for the uncoupled estimate is only of the size $2N_S =30$, whereas the number of coupled observations is $2^{N_S}=32768$. The improved observation model thus comes at an increased computational cost, which is exponential in the number of sensors, and becomes infeasible for larger numbers of sensors.

\section{Conclusion}

In this work we consider multi-marginal optimal transport problems with cost functions that decouple according to a tree, and we show that the entropy regularized formulation of this problem is equivalent to a generalization of a time and space discrete Schr\"odinger bridge defined on the same tree.
Moreover, we derive an efficient algorithm for solving this problem. 
We also compare the multi-marginally regularized optimal transport problem to a commonly used pairwise regularized optimal transport problem and illustrate the benefits 
 in theory and practice. Finally, we describe how to apply the framework to the problem of tracking an ensemble of indistinguishable agents.

Interestingly, the construction of the marginals in Theorem~\ref{thm:multi_omt_tree} is of the same form as the belief propagation algorithm \cite{Yedidia03,teh2002propagation} for inference in graphical models. In this interpretation, the vectors $u_j$ correspond to the local evidence in node $j\in\ccV$ and $\alpha_{(j_1,j_2)}$ is the message passed from node $j_2$ to node $j_1$.
Moreover, the objective function in \eqref{eq:A_bethe} can be interpreted as the Bethe free energy \cite{yedida05}, which is connected to belief propagation \cite{Yedidia03}.
These similarities of our framework to belief propagation algorithms
are investigated in \cite{haasler2020pgm}, and
 could be a stepping stone to extending our framework to graphs with cycles. 
Another direction of interest is the extension to continuous state models.

\appendix

\section{Proofs} \label{sec:appendix_proofs}
For an index set $\Gamma$, we denote the double sum $\sum_{j\in \Gamma}\sum_{i_j}$ by $\sum_{i_j : j\in \Gamma}$.
For simplicity of notation, we let $\ccM$ denote the set of matrices $M^{(j_1,j_2)}$, for $(j_1,j_2) \in \ccE$, and similarly write $\bm\mu$ and $\bm\lambda$ for the respective sets of optimization variables.
The proof of Theorem~\ref{thm:multi_omt_tree} is based on the following lemma.

\begin{lemma}[\!\!{\cite[Lemma 1 and 2]{elvander19multi}}] \label{lem:proj}
	Let $\bU= u_1 \otimes u_2 \otimes \dots \otimes u_J$, for a set of vectors $u_1, u_2,\dots,u_J$, and $\bK$ be a tensor of the same size.
	If
	$\langle \bK, \bU \rangle = w^T  u_j$,
	for a vector $w$ that does not depend on $u_j$, then it holds that
	$P_j(\bK \odot \bU ) = w \odot u_j$.

	Similarly, if
	$\langle \bK, \bU \rangle = w^T \diag( u_{j_1} ) W \diag(u_{j_2}) \hat w$,
	for two vectors $w,\hat w$, and a matrix $W$ that does not depend on $u_{j_1}$ and $u_{j_2}$, then it holds that $P_{j_1,j_2}(\bK \odot \bU ) = \diag(w \odot u_{j_1} ) W \diag(u_{j_2} \odot \hat w)$.
\end{lemma}

\begin{proof}[Proof of Theorem~\ref{thm:multi_omt_tree}]
	Due to the decoupling of the cost tensor $\bC$ as in \eqref{eq:cost_tensor_tree}, the tensor $\bK = \exp( - \bC/\epsilon)$ decouples  according to \eqref{eq:Kmultimarginal}
	with the matrices $K^{(j_1,j_2)}=\exp(C^{(j_1,j_2)}/\epsilon)$, for $(j_1,j_2)\in \ccE$. We can therefore write
	\begin{equation}
	\langle \bK, \bU \rangle =   \sum_{i_{j}:j\in\ccV} \bigg( \prod_{(j_1,j_2)\in \ccE} K^{(j_1,j_2)}_{i_{j_1},i_{j_2}} \bigg) \bigg( \prod_{k \in \ccV} (u_{k})_{i_{k}} \bigg) 
	=  \sum_{i_j} (u_j)_{i_j} (w_j)_{i_j},
	\end{equation}	
	where the vector $w_j$ is of the form
	\begin{equation}
	\begin{aligned}
		(w_j)_{i_\ell} =&  \sum_{i_k: k \in \ccV \setminus j } \bigg( \prod_{(j_1,j_2)\in \ccE} K^{(j_1,j_2)}_{i_{j_1},i_{j_2}} \bigg) \bigg( \prod_{j \in \ccV, j \neq \ell } (u_{j})_{i_{j}} \bigg). \\
	\end{aligned}
	\end{equation}
Consider the underlying tree to be rooted in node $j$. Then this can be written as	
$(w_j)_{i_\ell} = \prod_{ k: k \in \ccN_{j}  } \left( \alpha_{(j,k)}\right)_{i_j},$
where $\alpha_{(p(k),k)}\in \mR^n$ is defined by
	\begin{equation} \label{eq:alpha_inproof}
	\left( \alpha_{(p(k),k)}\right)_{i_{p(k)}} = \sum_{i_\ell: \ell\geq k} \prod_{ m>k } K^{(p(m),m)}_{i_{p(m)},i_{m}}  (u_m)_{i_m}.
	\end{equation}
If $k\in\ccL$, then
\begin{equation}
\left( \alpha_{(p(k),k)}\right)_{i_{p(k)}} = \sum_{i_k}  K^{(p(k),k)}_{i_{p(k)},i_{k}}  (u_k)_{i_k} =\left( K^{(p(k),k)} u_k \right)_{i_{p(k)}}.
\end{equation}
Otherwise, it holds
\begin{equation}
\begin{aligned}
\left( \alpha_{(p(k),k)}\right)_{i_{p(k)}}  
&= \sum_{i_k} \bigg( K^{(p(k),k)}_{i_{p(k)},i_{k}}  (u_k)_{i_k}  \prod_{ \substack{ \ell \in \ccN_k \\ \ell \neq p(k)}} \!\! \Big(  \! \sum_{i_m : m \geq \ell } \prod_{ j > \ell } K^{(p(j),j)}_{i_{p(j)},i_{j}}  (u_j)_{i_j} \Big) \bigg) \\
&= \sum_{i_k} \bigg( K^{(p(k),k)}_{i_{p(k)},i_{k}}  (u_k)_{i_k} \prod_{ \substack{ \ell \in \ccN_k \\ \ell \neq p(k)}} \left( \alpha_{(p(\ell,\ell)} \right)_{i_{p(\ell)}} \bigg).
\end{aligned}
\end{equation}
This inductively defines the vectors $\alpha_{(j,k)}$ as in \eqref{eq:alpha}. The expression for the projection follows from Lemma~\ref{lem:proj} with $w_j = \bigodot_{k\in \ccN_j} \alpha_{(j,k)}$.
\end{proof}

\begin{proof}[Proof of Proposition~\ref{prp:GammaL}]
	Note that the optimal solution to \eqref{eq:omt_multi_regularized} with cost structured according to the tree $\ccT=(\ccV,\ccE)$ is of the form $\bM= \bK \odot \bU$, where $\bK$ is defined by \eqref{eq:Kmultimarginal} and $\bU$ is defined by \eqref{eq:U}.
	To prove the first claim, let $k\in \ccV$ lie on the direct path between $j\in \ccV$ and $\ell \in \ccV$.
	Then straightforward computation implies that
	\begin{equation}
	\left(P_{jk\ell} (\bK \odot \bU )\right)_{i_j,i_k,i_\ell} \left( P_{k} (\bK \odot \bU)\right)_{i_k}  = \left(P_{jk} (\bK \odot \bU )\right)_{i_j i_k} \left(P_{k\ell} (\bK \odot \bU)\right)_{i_k,i_\ell}.
	\end{equation}
	In particular, for any fixed $i_k$, the matrix
	$\left(P_{jk\ell} (\bK \odot \bU )\right)_{\cdot,i_k,\cdot}$
	is of rank 1. Hence, given $P_k(\bK \odot \bU)$, the marginals $P_j(\bK \odot \bU)$ and $P_\ell(\bK \odot \bU)$ have no influence on each other. Thus, the optimal transport problem \eqref{eq:omt_multi_regularized} on $\ccT$ can be decoupled into smaller problems, by cutting $\ccT$ in $k\in \ccV$, and solving an optimal transport problem \eqref{eq:omt_multi_regularized} on each of the resulting subtrees, where $k\in \Gamma$ and $k\in \ccL$ for each of the subtrees.
	
	To prove the second claim, recall from the definition of the tensor $\bU$ in \eqref{eq:U} that $u_j=\ett$ for all $j\in\ccV\setminus\Gamma$.
	Thus, for $k$ such that $(k,\ell)\in\ccE$, the corresponding vector \eqref{eq:alpha} is $\alpha_{(k,\ell)}= K^{(k,\ell)} \ett$, which is constant and does not need to be updated when recomputing the projections \eqref{eq:proj_j}.
	It thus suffices to solve \eqref{eq:omt_multi_regularized} on the subtree of $\ccT$ that is obtained by removing $\ell \in \ccV$ and $(k,\ell) \in \ccE$ from $\ccT$.
\end{proof}

\begin{proof}[Proof of Theorem~\ref{thm:HMM_tree_sol}]
	Assume that the only edge with node $j=1$ is denoted $(1,2)$.
	We add the trivial constraints $M^{(j_1,j_2)} \ett = (M^{(p(j_{1}),j_1)})^T \ett$
	 for $(j_1,j_2)\in\ccE_r\setminus (1,2)$. We relax this constraint, together with the constraint on $\mu_1$, and let $\lambda_{(j_1,j_2)}$, for $(j_1,j_2)\in\ccE_r$, denote the corresponding dual variables.
	Furthermore, we relax the constraints $(M^{(j_p,j)})^T \ett = \mu_j$ with dual variables $\lambda_j$, for the leaves $j\in\ccL \setminus 1$. A Lagrangian for \eqref{eq:HMM_tree} is then
	\begin{equation}
	\begin{aligned}
	  L( \ccM,\bm\mu,\bm\lambda) &=  \sum_{(j_1,j_2) \in \ccE}  H \left( M^{(j_1,j_2)} | \diag( \mu_{j_1}) A^{(j_1,j_2)} \right)   + \lambda_{(1,2)}^T ( M^{(1,2)} \ett - \mu_1) \\
	 & \!\!\!\!\!\!\! + \sum_{  j_2 \in \ccL} \lambda_{j_2}^T ( \mu_{j_2} - (M^{(j_1,j_2)})^T \ett ) + \! \sum_{\substack{(j_1,j_2) \in \ccE \\ j_1 \neq 1 } }  \lambda_{(j_1,j_2)}^T ( M^{(j_1,j_2)} \ett - (M^{(j_{1_p}, j_1)})^T \ett ).
	\end{aligned}
	\end{equation}
	When $j_2$ is an inner node on the tree, i.e., $j_2\notin \ccL$, the derivative with respect to the entries $M^{(j_1,j_2)}_{i_1 i_2}$, for all $i_1,i_2=1,\dots,n$, and $(j_1,j_2) \in \ccE_r$ is
	\begin{equation} \label{eq:Lagrange_der}
	\log\Bigg( \frac{M^{(j_1,j_2)}_{i_1 i_2}}{(\mu_{j_1})_{i_1} A^{(j_1,j_2)}_{i_1 i_2}} \Bigg)   +  (\lambda_{(j_1,j_2)})_{i_2} -  \sum_{ k: (j_2,k) \in \ccE } (\lambda_{(j_2,k)})_{i_2}.
	\end{equation}
	Since \eqref{eq:HMM_tree} is convex, a mass transport plan is optimal if this gradient vanishes, yielding the expression in terms of the other variables,
		\begin{equation} \label{eq:Mj1j2}
	M^{(j_1,j_2)} = \diag(\mu_{j_1} ./ v_{(j_1,j_2)})  A^{(j_1,j_2)} \diag\bigg( \bigodot_{k : (j_2,k) \in \ccE} \!\!\!  v_{(j_2,k)} \bigg),	
	\end{equation}
	where $v_{(j_1,j_2)} = \exp(\lambda_{(j_1,j_2)})$ for all $(j_1,j_2)\in \ccE_r$.
	In case $j_2\in\ccL$, the last sum in the derivative \eqref{eq:Lagrange_der} is replaced by $ (\lambda_{j_2})_l$. Defining $v_{j_2}=\exp(\lambda_{j_2})$, the optimal mass transport plan is thus of the form
	\begin{equation} \label{eq:Mj1j2_leaf}
	M^{(j_1,j_2)} = \diag( \mu_{j_1} ./ v_{(j_1,j_2)})  A^{(j_1,j_2)} \diag( v_{j_2} ).
	\end{equation}
	Next, note that the marginal of the optimal transport plan satisfies
 \begin{align} \label{eq:muj2}
&\mu_{j_2} = M_{(j_1,j_2)}^T \ett = \hat \varphi_{j_2} \odot \varphi_{j_2}, \\
\text{with} \quad & \varphi_{j_2} = \bigodot_{k:(j_2,k) \in \ccE} v_{(j_2,k)},
\quad \mbox{ and }\quad 
\hat \varphi_{j_2} = (A^{(j_1,j_2)})^T \left( \mu_{j_1} ./ v_{(j_1,j_2)} \right). \label{eq:phij}
\end{align} 
	Since for all $k$ such that $(j_2,k)\in \ccE_r$ it holds that $M_{(j_2,k)}\ett = \mu_{j_2}$, we get
	\begin{equation} \label{eq:vjl}
	v_{(j_2,k)} =  A^{(j_2,k)} \bigg( \bigodot_{\ell: (k,\ell) \in \ccE} v_{(k,\ell)} \bigg) =  A^{(j_2,k)}  \varphi_{k},
	\end{equation}
	which completes the definition of the vectors $(\varphi_j)_{j\in \ccV}$.
	Similarly to \eqref{eq:muj2} it holds that
	\begin{equation} \label{eq:muj1}
	\mu_{j_1} = \bigg( \bigodot_{k:(j_1,k) \in \ccE} v_{(j_1,k)} \bigg) \odot  A \left( \mu_{j_{\parent(1)}} ./ v_{(j_{\parent(1)},j_1)} \right).
	\end{equation}
	Plugging \eqref{eq:muj1} into $\hat \varphi_{j_2}$ in \eqref{eq:phij}, yields the recursive definition of the vectors $(\hat \varphi_j)_{j \in \ccV}$.
	The expression for the mass transport plans $M^{(j_1,j_2)}$, for $(j_1,j_2)\in \ccE$, in terms of the vectors $\hat \varphi_{j_1}$, $\varphi_{j_2}$ and $\varphi_{j_1\setminus j_2}$ follows by identifying them in expressions \eqref{eq:Mj1j2} and \eqref{eq:Mj1j2_leaf}.
\end{proof}

\begin{proof}[Proof of Proposition~\ref{prp:HMM_tree_anyroot}]
	For any matrices $M,A\in \RR^{n\times n}_+$ and $\mu\in \RR^n_+$ one can write the term $ H\left( M \,|\, \diag(\mu) A \right)$ as
	\begin{equation} \label{eq:KL_two_sums}
	\sum_{i,j=1}^n \left( M_{ij} \log \left( \frac{M_{ij}}{A_{ij}}  \right) - M_{ij} \right) + \sum_{i,j=1}^n  \left( \mu_i A_{ij} - M_{ij} \log(\mu_i) \right).
	\end{equation}
	Since $M \ett = \mu$ and $A \ett = \ett$, the second sum can be simplified to
	$ \sum_{i=1}^n  \left( \mu_i  - \mu_i \log(\mu_i) \right)$.
	Furthermore, adding the term
	$\sum_{i,j=1}^n A_{ij} - \sum_{i=1}^n 1 = 0$
	to \eqref{eq:KL_two_sums}, the expression can be written as
	$H\left( M \,|\,  A \right) - H( \mu )$. 
	Due to the underlying tree structure of problem \eqref{eq:HMM_tree}, the number of outgoing edges from the root node $j_r$ is $\deg(j_r)$, and for all other vertices $j\in\ccV \setminus \{j_r\}$ the number of outgoing edges is $\deg(j)-1$. In the case that $j_r\in\ccL$, since the marginal $\mu_{j_r}$ is known, the term $H(\mu_{j_r})$ is constant and can be removed from the objective without changing the optimal solution.
\end{proof}

\begin{proof}[Proof of Theorem~\ref{thm:equivalence}]
	
	Note that for a rooted directed tree, in each node $j\in \ccV \setminus \ccL$, there is one incoming edge and the rest of the connected edges are outgoing. Its neighbouring nodes are therefore given by the set $\ccN_j = \parent(j) \cup \left\{ k : (j,k) \in\ccE_r \right\}$.
	Thus,
	\begin{equation} \label{eq:alpha_prod}
	\bigodot_{k\in \ccN_j} \alpha_{(j,k)} = (K^{(\parent(j),j)})^T \alpha_{(j,\parent(j))} \odot \bigodot_{k:(j,k) \in\ccE_r} K^{(j,k)} \alpha_{(j,k)}.
	\end{equation}	
	For any edge $(j,k) \in \ccE$ and for the reverse edge $(j,j_p)$ it holds that
\begin{align}
	\alpha_{(j,k)} &= \bigodot_{ \ell \in \ccN_k \setminus \{j\} } K^{(k,\ell)} \alpha_{(k,\ell)} =  \bigodot_{\ell:(k,\ell) \in\ccE} K^{(k,\ell)} \alpha_{(k,\ell)}\\
\alpha_{(j,\parent(j))} &= \!\!\!\!\! \!\!\! \bigodot_{ \ell \in \ccN_{\parent(j)} \setminus \{j\} } \!\!\!\!\! \!\! K^{(\parent(j),\ell)} \alpha_{(j,\ell)}
	\! = \!(K^{(\parent(\parent(j)),\parent(j))})^T \alpha_{(\parent(j),\parent(\parent(j))} \odot  \!\!\!\!\! \bigodot_{\substack{ \ell:(\parent(j),\ell) \in\ccE,\\ \ell \neq j}}  \!\!\!\!\! K^{(\parent(j),\ell)} \alpha_{(\parent(j),\ell)}.
\end{align}
	By associating the first term in \eqref{eq:alpha_prod} with $\hat \varphi_j$ and the second term with $ \varphi_j$, we see that the tensor structure of $\bK$ gives rise to the construction of $\varphi_j$ and $\hat \varphi_j$ as in Theorem~\ref{thm:HMM_tree_sol}. Hence, we can define a tensor in analogy to $\bU$ of the form $\bV= (\ett./v_1) \otimes v_2 \otimes \dots \otimes v_J$, where $v_j$ are the vectors from Theorem~\ref{thm:HMM_tree_sol} if $j\in \ccL$, and $v_j=\ett$ if $j\in \ccV \setminus \ccL$. Then, the tensor $\bK \odot \bV$ has the same marginals as the tensor $\bK \odot \bU$. Due to an extension of Sinkhorn's theorem to tensors \cite{franklin1989}, it follows $\bU = \bV$ (up to scaling with a factor and its inverse in the vectors $u_1,\dots,u_J$), and $P_j(\bM) = \mu_j$ for all $j\in \ccV$.
\end{proof}

\begin{proof}[Proof of Proposition~\ref{prp:HMM_sinkhorn}]
	Based on the proof to Theorem~\ref{thm:HMM_tree_sol}, a Lagrange dual to problem \eqref{eq:HMM_tree} can be formulated as to maximize
	\begin{equation} \label{eq:dual_HMM_tree}
	\begin{aligned}
	&- \!\!\!\! \sum_{(j_1,j_2)\in \ccE} \!\!\!\! \left(\mu_{j_1} \odot \exp(-\lambda_{(j_1,j_2)}) \right)  A^{(j_1,j_2)}  \Big( \!\!\! \bigodot_{k:(j_2,k)\in \ccE} \!\!\!\!\! \exp(\lambda_{(j_2,k)}) \Big)  \\
	& - \sum_{j \in\ccL}   \left(\mu_{j_p} \odot \exp(-\lambda_{(j_p,j)}) \right) A^{(j_1,j_2)} \left(  \exp(\lambda_{j}) \right)   - \lambda_{(1,2)}^T \mu_1 + \sum_{j \in \ccL} \lambda_j^T \mu_j
	\end{aligned}
\end{equation}
	with respect to $\mu_j$, for all inner nodes $j$, and the dual variables $\lambda_{(j_1,j_2)}$, for $(j_1,j_2)\in \ccE$, and $\lambda_j$, for $j\in\ccL$.
	A block coordinate ascent in the dual is then to iteratively maximize \eqref{eq:dual_HMM_tree} with respect to one of the dual variable vectors, while keeping the others fixed. 
	Denote $v_{(j_1,j_2)}=\exp(\lambda_{(j_1,j_2)})$, for $(j_1,j_2)\in \ccE$, and ${v_j = \exp(\lambda_j)}$, for $j\in \ccL$.
	The gradient of \eqref{eq:dual_HMM_tree} with respect to $\lambda_{(1,2)}$ vanishes if it holds that
	\begin{equation}
	v_{(1,2)}  =A^{(j_1,j_2)} \bigg( \bigodot_{k:(2,k)\in\ccE} v_{(2,k)} \bigg),
	\end{equation}
	and the gradient with respect to $\lambda_j$ vanishes if $v_j = \mu_j ./ ( (A^{(j_p,j)})^T (\mu_{j_p}./ v_{(j_p,j)} ) )$, for $j\in\ccL$, and where $j_p\in \ccV$ is the parent of node $j$.
	Finally, the gradient of \eqref{eq:dual_HMM_tree} with respect to $\lambda_{(j_1,j_2)}$, where $j_1\neq 1$, vanishes if
	$M_{(j_1,j_2)} \ett = M_{(j_p,j_1)}^T \ett$,
	where $j_p$ is the parent of node $j_1$,
	which leads to the recursive definition of $\varphi_j$ and $\hat \varphi_j$, for $j\in \ccV$, as in Theorem~\ref{thm:HMM_tree_sol}. Hence, given an initial set of positive vectors $v_{(1,2)}$ and $v_j$, for $j\in\ccL$, the scheme \eqref{eq:sinkhorn_hmm} is a block coordinate ascend in a Lagrange dual of problem \eqref{eq:HMM_tree}.	
\end{proof}

\bibliographystyle{siamplain}

\bibliography{ref}



\section{Supplementary material}


\begin{proof}[Proof of Proposition~\ref{prp:multi_omt_tree_pairwise}]
	Let $r\in\ccL$ and assume that $j_1$ lies on the path from $r$ to $j_L$ (note that $r=j_1$ if $j_1\in\ccL$).
	For the pairwise marginals we would like to write the inner product   as in Lemma~\ref{lem:proj}:
	\begin{equation}
	\begin{aligned}
	\langle \bK, \bU \rangle = &\sum_{i_{j}:j\in\ccV} \bigg( \prod_{(k_1,k_2)\in \ccE_r} K^{(k_1,k_2)}_{i_{k_1},i_{k_2}} \bigg) \bigg( \prod_{j \in \ccV} (u_{j})_{i_{j}} \bigg) \\
	= & \sum_{i_{j_1},i_{j_L}} (w_{j_1})_{i_{j_1}} (u_{j_1})_{i_{j_1}} W_{i_{j_1}i_{j_L}} (u_{j_L})_{i_{j_L}} (\hat w_{j_L})_{i_{j_L}}.
	\end{aligned}
	\end{equation}
	Note that the nodes $\ccV\setminus\{j_1, j_L\}$ can be partitioned into three sets, which are separated by the nodes $j_1$ and $j_L$: $\{j\in \ccV:j \ngeq j_2, j\neq j_1\}$, $\{j\in \ccV: j>j_1, j\ngeq j_L\}$ and  $\{j\in \ccV:j>j_L \}$.	Then $w_{j_1}$ corresponds to the contribution from the first set 
	\begin{equation}
	(w_{j_1})_{i_{j_1}} = \!\! \sum_{i_k : k \ngeq j_2, k\neq j_1} \!\! \bigg( \! \prod_{ \substack{(k_1,k_2)\in \ccE_r  k_2 \ngeq j_2 }} \! \!\!  K^{(k_1,k_2)}_{i_{k_1},i_{k_2}} \bigg) \bigg( \! \prod_{ \substack{k \in \ccV\\ k\ngeq  j_2, k\neq j_1 }} \!\!\! (u_{k})_{i_{k}} \bigg) 
	=  \prod_{k \in \ccN_{j_1} \setminus j_2} \left(\alpha_{(j_1,k)}\right)_{i_{j_1}} \!,
	\end{equation}
	and similarly $\hat w_{j_L}$ is the contribution from the third set
	\begin{equation}
	(\hat w_{j_L})_{i_{j_L}} =  \sum_{i_k : k> j_L} \bigg( \prod_{ \substack{(k_1,k_2)\in \ccE \\ k_2> j_L } }   K^{(k_1,k_2)}_{i_{k_1},i_{k_2}} \bigg) \bigg( \prod_{ \substack{k \in \ccV\\ k> j_L } } (u_{k})_{i_{k}} \bigg) 
	= \prod_{k \in \ccN_{j_L} \setminus j_{L-1}} \left( \alpha_{(j_L,k)} \right)_{i_{j_L}}.
	\end{equation}
	The matrix $W$ is then obtained by summing over all the indices corresponding to the second index set $\{i_k : k> j_1, k \ngeqslant j_L\}$
	\begin{equation}
	\begin{aligned}
	W_{i_{j_1} i_{j_L}} =& \!\! \sum_{i_k : k> j_1, k \ngeqslant j_L} \prod_{\substack{(k_1,k_2)\in \ccE \\ k_2>j_1, k_1 \ngtr j_L}} K^{(k_1,k_2)}_{i_{k_1},i_{k_2}} \prod_{\substack{k \in \ccV\\ k>j_1,k \ngeqslant j_L}} (u_{k})_{i_{k}}\\
	=& \sum_{i_{j_2},\dots,i_{j_{L-1}}} \Bigg( \prod_{\ell=2}^{L-1} \bigg( K^{(j_{\ell-1},j_\ell)}_{i_{j_{\ell-1}},i_{j_\ell}} (u_{j_\ell})_{i_{j_\ell}} \!\! \prod_{ \substack{k\in \ccN_{j_\ell}\\ k \neq j_{\ell-1}, j_{\ell+1}}} \!\! ( \alpha_{(j_\ell,k)})_{i_{j_\ell}} \bigg)  K^{(j_{L-1},j_L)}_{i_{j_{L-1}},i_{j_L}} \Bigg) \\
	=& \Bigg( \bigg( \prod_{\ell=2}^{L-1}  K^{(j_{\ell-1},j_\ell)} \diag \Big( u_{j_\ell} \odot \!\!\!\! \bigodot_{ \substack{k\in \ccN_{j_\ell}\\ k \neq j_{\ell-1}, j_{\ell+1}}} \!\!\!\! ( \alpha_{(j_\ell,k)}) \Big) \bigg)  K^{(j_{L-1},j_L)} \Bigg)_{i_{j_1},i_{j_L}}.
	\end{aligned}
	\end{equation}
	Applying Lemma~\ref{lem:proj}, we get the projection on the marginals $j_1$ and $j_L$.
\end{proof}

\begin{proof}[Proof of Corollary~\ref{cor:root_independent}]
	Let the set of matrices $M^{(j_1,j_2)}$, for $(j_1,j_2)\in \ccE_r$, and vectors $\mu_j$, for $j\in \ccV \setminus \ccL$, be the solution to \eqref{eq:HMM_tree} on the rooted directed tree $\ccT_r=(\ccV,\ccE_r)$. It is thus also the solution to \eqref{eq:HMM_tree_anyroot} on the same tree.	
	Changing the root to another vertex $\hat r \in \ccL$, requires switching the direction of all edges on the direct path between $r$ and $\hat r$.	
	This is done by replacing the respective transition probability matrix $ A^{(j_1,j_2)}$ by the reverse transition probability matrix $A^{(j_2,j_1)}$, which is given by $	A^{(j_2,j_1)} = \diag( \ett./ a_{j_2}) \big( A^{(j_1,j_2)} \big)^T \diag(a_{j_1}),$ 
	for two given vectors $a_{j_1}$ and $a_{j_2}$ \cite{pavon2010discrete}.
	Note that
	\begin{equation}
	\begin{aligned}
	&H\left( M^{(j_2,j_1)} \,|\, A^{(j_2,j_1)} \right) 
	= H\left( M^{(j_2,j_1)} \,|\, \diag(\ett ./ a_{j_1}) (A^{(j_1,j_2)})^T \diag( a_{j_2}) \right) \\
	& \qquad \qquad \qquad   = H\left( (M^{(j_1,j_2)})^T \,|\, A^{(j_1,j_2)} \right) + H( \mu_{j_1} | \diag( \ett./ a_{j_1} ) )  + H(\mu_{j_2} | \diag( a_{j_2} ).
	\end{aligned}
	\end{equation}
	Summing over these terms for all edges on the path between $r$ and $\hat r$, the last two terms cancel for all inner nodes $j_1,j_2\in\ccV\setminus \Gamma$, and are constants for $j_1,j_2\in \{ r, \hat r \}$.
	From problem \eqref{eq:HMM_tree_anyroot} we thus see that the optimal distributions $\mu_j$, for $j\in \ccV$, are unchanged, and for the transport plans on the reversed edges it holds $M^{(j_2,j_1)}= (M^{(j_1,j_2)})^T$.	
\end{proof}

\end{document}